\documentclass[12pt]{amsart}
\usepackage{amsmath,amssymb,latexsym,euscript}
\usepackage{mathrsfs}
\usepackage{enumerate}
\usepackage{color}
\usepackage[colorlinks=true, pdfstartview=FitV, linkcolor=red,%
citecolor=green, urlcolor=blue, pagebackref]{hyperref}
\usepackage[all, knot]{xy}
\xyoption{arc}
\usepackage{bbm}

\numberwithin{equation}{section}

\usepackage{url}
\usepackage{longtable,tabu}
\usepackage{float}

\usepackage[latin1]{inputenc}
\usepackage{xspace,amsfonts}
\usepackage{amsthm}
\input xy \xyoption {all}

\usepackage{tikz,environ}
\usetikzlibrary{arrows}
\usetikzlibrary{patterns,snakes}

\def\veps{\hat{\varepsilon}}

\def\RSS{{\mathscr R}}

\newcommand\RR{\mathscr{R}}
\newcommand\R[1][n]{\RR_{#1}^{\Lambda}}

  \def\hg{{\mathfrak h}}

       \def\km{{\mathbbm k}}

    \def\NM{{\mathbb{N}}}

    \def\QM{{\mathbb{Q}}}
    \def\RM{{\mathbb{R}}}
\def\SG{{\mathfrak S}}

    \def\ZM{{\mathbb{Z}}}

\def\bnu{\mathbf{\nu}}

  \def\kb{{\mathbf k}}

    \def\OC{{\mathcal{O}}}

\def\Lam{\Lambda}
\def\lam{\lambda}
\newcommand{\nc}{\newcommand} \newcommand{\renc}{\renewcommand}

\renc{\l}{\lambda}

\newcommand{\rdots}{\mathinner{ \mkern1mu\raise1pt\hbox{.}
    \mkern2mu\raise4pt\hbox{.}
    \mkern2mu\raise7pt\vbox{\kern7pt\hbox{.}}\mkern1mu}}

\DeclareMathOperator{\Ann}{Ann}

\DeclareMathOperator{\Ker}{Ker}

\def\ov{\overline}
\def\un{\underline}

\def\to{\rightarrow}

\def\laongto{\laongrightarrow}

\nc{\triright}{\stackrel{[1]}{\to}}
\nc{\laongtriright}{\stackrel{[1]}{\laongto}}

\DeclareMathOperator{\Hom}{Hom}

\numberwithin{equation}{section}
\newtheorem{prop}[equation]{Proposition}
\newtheorem{thm}[equation]{Theorem}
\newtheorem{cor}[equation]{Corollary}

\newtheorem{lem}[equation]{Lemma}

\theoremstyle{definition}
\newtheorem{dfn}[equation]{Definition}

\newtheorem{cconj}[equation]{Center Conjecture}

\theoremstyle{remark}
\newtheorem{rem}[equation]{Remark}

\theoremstyle{remark}
\newtheorem{remark}[equation]{Remark}

\newcommand{\ra}{\rightarrow}

\newcommand{\into}{\hookrightarrow}

\def\del{\partial}

\nc{\simto}{\stackrel{\sim}{\to}}

\nc{\la}{\langle}
\renc{\ra}{\rangle}

\nc{\diag}{\mathrm{diag}}
\newcommand{\ubr}[2]{\underbrace{#1}_{#2}}

\makeatletter
\newcommand\leftdash{\!\rotatebox[origin=c]{-60}{$\dabar@\dabar@\dabar@$}\!}
\newcommand\rightdash{\!\rotatebox[origin=c]{60}{$\dabar@\dabar@\dabar@$}\!}
\makeatother

\newlength{\my}
\title[Center conjecture for the cyclotomic KLR algebras]{On the center conjecture for the cyclotomic KLR algebras}

\author{Jun Hu}\address{MIIT Key Laboratory of Mathematical Theory and Computation in Information Security\\
  School of Mathematical and Statistics\\
  Beijing Institute of Technology\\
  Beijing, 100081, P.R. China}
\email{junhu404@bit.edu.cn}

\author{Huang Lin}\address{School of Mathematical Sciences\\
 Zhejiang University\\
 Hangzhou, 310027, P.R. China}
\email{3130100651@zju.edu.cn}

\begin{document}



\begin{abstract}
The center conjecture for the cyclotomic KLR algebras $\R[\beta]$ asserts that the center of $\R[\beta]$ consists of symmetric elements in its KLR $x$ and
$e(\nu)$ generators. In this paper we show that this conjecture is equivalent to the  injectivity of some natural map $\overline{\iota}_\beta^{\Lam,i}$ from the cocenter of
$\RR_\beta^{\Lam}$ to the cocenter of $\RR_\beta^{\Lam+\Lam_i}$ for all $i\in I$ and $\Lam\in P^+$. We prove that the map $\overline{\iota}_\beta^{\Lam,i}$
is given by multiplication with a center element $z(i,\beta)\in\mathscr{R}_{\beta}^{\Lambda+\Lambda_i}$ and we explicitly calculate the element $z(i,\beta)$
in terms of the KLR $x$ and $e(\nu)$ generators. We present an explicit monomial basis for certain bi-weight spaces of the defining ideal of $\R[\beta]$ and of
$\R[\beta]$. For $\beta=\sum_{j=1}^n\alpha_{i_j}$ with $\alpha_{i_1},\cdots,
\alpha_{i_n}$ pairwise distinct, we construct an explicit monomial basis of $\R[\beta]$, prove the map $\overline{\iota}_\beta^{\Lam,i}$ is injective and thus verify the center conjecture for these $\mathscr{R}_{\beta}^{\Lambda}$.
\end{abstract}

\maketitle


\section*{Introduction}

Given a Catan datum or a symmetrizable Cartan matrix, Khovanov-Lauda (\cite{KL09,KL11}) and Rouquier (\cite{R08}) introduced a vast family of algebras $\RR_\beta$, called KLR algebras or quiver Hecke algebras, and use them to provide a categorification of the negative part of the quantum groups associated to the same Cartan datum. These algebras play an important role in the categorical representation theory of $2$-Kac-Moody algebras. Khovanov, Lauda and Rouquier also defined their graded cyclotomic quotients $\R[\beta]$ associated with an integral dominant weight $\Lambda$ and conjectured that they categorify the corresponding irreducible integrable highest weight module of the quantum groups. This conjecture was later proved by Kang-Kashiwara (\cite{KK12}) and Webster (\cite{Web10}).

The KLR algebras generalize the affine Hecke algebras in many aspects of structure and representation theory. For example, both of them admit a faithful polynomial representation and have monomial bases.
The center of a KLR algebra is proved to consist of all symmetric elements in its KLR generators $x_1,\ldots ,x_n$ and $e(\nu)$, which is similar to the well-known Bernstein theorem on the centers of affine Hecke algebras.
In the case of type $A^{(1)}_{e}$ and $A_{\infty}$, Brundan and Kleshchev (\cite{BK08}) proved these cyclotomic quotients $\R[\beta]$ are isomorphic to the block algebras of the cyclotomic Hecke algebras of type $A$. The graded cellular bases of the cyclotomic quotients $\R[\beta]$ have been constructed in \cite{HM}, \cite{Li}, \cite{MT1}, \cite{MT2} in some finite and affine types. For more results on the structure and representation theory of the cyclotomic KLR algebras $\R[\beta]$, see \cite{AP14, AP16a, AP16b, APS, HS, Sp17}.

One of the major unsolved open problems on $\R[\beta]$ is the center conjecture, which asserts that the centers of $\R[\beta]$ consists of symmetric elements in
its KLR $x$ and $e(\nu)$ generators. Using Brundan-Kleshchev's isomorphism, one can reduce the center conjecture for $\R[\beta]$ to the claim that the center of
the cyclotomic Hecke algebra consists of symmetric polynomials in its Jucys-Murphy operators. The latter was proved for the degenerated cyclotomic Hecke algebras
by Brundan (\cite{B08}). In the level one case, the latter is just the Dipper-James conjecture (\cite{DJ87}) for the Iwahori-Hecke algebras of type $A$, which
was proved by Francis and Graham (\cite{FG06}).

In this paper we show that the center conjecture for general $\R[\beta]$ is equivalent to the claim that the natural homomorphism $\overline{p}_\beta^{\Lam,i}$
from the center of $\RR_\beta^{\Lam+\Lam_i}$ to the center of $\RR_\beta^{\Lam}$ is surjective for any $i\in I$ and any $\Lam\in P^+$  (which was first
observed in \cite{H17}). The latter is clearly equivalent to the injectivity of the naturally induced dual map $\overline{\iota}_\beta^{\Lam,i}$ from the
cocenter $\RR_\beta^{\Lam}/[\RR_\beta^{\Lam},\RR_\beta^{\Lam}]$ of $\RR_\beta^{\Lam}$ to the cocenter $\RR_\beta^{\Lam+\Lam_i}/[\RR_\beta^{\Lam+\Lam_i},
\RR_\beta^{\Lam+\Lam_i}]$ of $\RR_\beta^{\Lam+\Lam_i}$ for any $i\in I$ and any $\Lam\in P^+$. We show that the map $\iota_\beta^{\Lam,i}$ is given by
multiplication with a center element $z(i,\beta)$ and we explicitly calculate the element $z(i,\beta)$ in terms of
the KLR $x$ and $e(\nu)$ generators. In the case of type $A^{(1)}_{e}$ and $A_{\infty}$, the multiplication with the center element $z(i,\beta)$ actually
sends each graded cellular basis element (resp., monomial basis element) of $\R[\beta]$ to a graded cellular basis element (resp., monomial basis element) of
$\RR_\beta^{\Lam+\Lam_i}$. Thus the studying of the injectivity of $\overline{\iota}_\beta^{\Lam,i}$ seems to be a rather promising approach (which we named it
as ``cocenter approach'') as long as one can construct a basis of the cocenter represented by some graded cellular basis or monomial basis elements.
In this paper, we construct a monomial basis for certain bi-weight space of the defining ideal of $\R[\beta]$. For $\beta=\sum_{j=1}^n\alpha_{i_j}$ with
$\alpha_{i_1},
\cdots,\alpha_{i_n}$ pairwise distinct, we construct an explicit monomial basis for the cocenter of $\R[\beta]$. We introduce a new combinatorial notion of
``indecomposable relative to $\nu$''  in Definition \ref{indecomp0} which seems of independent interest. This notion plays a key role in our construction of the
monomial basis of the cocenter. We then apply the
``cocenter approach'' to show that in this case
the map $\overline{\iota}_\beta^{\Lam,i}$ is injective and thus verify the center conjecture for these $\mathscr{R}_{\beta}^{\Lambda}$.

This paper is organized as follows. In Section~\ref{sec:Prel}, we recall some basic definitions and properties of the KLR algebra $\RR_\beta$ and its cyclotomic
quotient $\R[\beta]$.
In Section~\ref{sec:DC}, we present our cocenter approach to the center conjecture for general $\R[\beta]$. We show that in the center conjecture is equivalent to
the injectivity of the naturally induced dual map $\overline{\iota}_\beta^{\Lam,i}$ from the cocenter $\RR_\beta^{\Lam}/[\RR_\beta^{\Lam},\RR_\beta^{\Lam}]$ of
 $\RR_\beta^{\Lam}$ to the cocenter $\RR_\beta^{\Lam+\Lam_i}/[\RR_\beta^{\Lam+\Lam_i},\RR_\beta^{\Lam+\Lam_i}]$ of $\RR_\beta^{\Lam+\Lam_i}$ for any $i\in I$
 and any $\Lam\in P^+$, and the map $\overline{\iota}_\beta^{\Lam,i}$ is given by multiplication with a center element $z(i,\beta)$.
 The main result is Theorem \ref{thm:annker}, where we explicitly calculate the element $z(i,\beta)$ in terms of the KLR $x$ and $e(\nu)$. The proof makes
 essential use of the explicit formulae for the bubbles presented in \cite[Definition A.1, Proposition B.3]{SVV}. In Section~\ref{sec:MBCKLR}, we present an
 explicit monomial basis (in Theorem \ref{thm:MBBWS} and Corollary \ref{maincor2}) for certain bi-weight spaces of the defining ideal of $\R[\beta]$ and
of the whole $\R[\beta]$. In Section~\ref{sec:MBCocen}, we study the cyclotomic KLR algebras when $\beta=\sum_{j=1}^n\alpha_{i_j}$
 with $\alpha_{i_1},\cdots,\alpha_{i_n}$ pairwise distinct. We introduce a combinatorial notion of ``indecomposable
relative to $\nu$''  in Definition \ref{indecomp0} and give a complete characterization of all indecomposable elements in $\SG_n$ which is
relative to a given $\nu\in I^\beta$ in Lemma \ref{indecomp1}. We construct in Theorem \ref{thm:MBCocen} an explicit monomial basis for the cocenter of $\R[\beta]$.
Finally, we show in Lemma \ref{finalInject} that in this case the map $\overline{\iota}_\beta^{\Lam,i}$ is injective, and thus obtain in Theorem \ref{thm:cenconj} that
the Center Conjecture \ref{conj:CC} holds for these $\mathscr{R}_{\beta}^{\Lambda}$.

\bigskip
\centerline{Acknowledgements}
\bigskip

The research was supported by the National Natural Science Foundation of China (No. 12171029).
\bigskip

\section{Preliminary}\label{sec:Prel}

In this section, we shall recall some basic definitions and properties of the KLR algebras and their cyclotomic quotients.

Let $I$ be an index set. An integral square matrix $A = (a_{i,j})_{i,j\in I}$ is called a \emph{symmetrizable generalized Cartan matrix} if it satisfies
\begin{enumerate}
  \item $a_{ii} = 2$, $\forall i \in I$;
  \item $a_{ij} \leqslant 0\, (i \neq j)$;
  \item $a_{ij} = 0 \Leftrightarrow \, a_{ji} = 0\, (i,j \in I)$;
  \item there is a diagonal matrix $D = \diag (d_i\in\ZM_{>0}\mid i\in I)$ such that $DA$ is symmetric.
\end{enumerate}

A Cartan datum $(A,P,\Pi , P^{\vee} , \Pi^{\vee})$ consists of
\begin{enumerate}
  \item a symmetrizable generalized Cartan matrix $A$;
  \item a free abelian group $P$ of finite rank, called the \emph{weight lattice};
  \item $\Pi = \{ \alpha_i \in P \mid i\in I \}$, called the set of \emph{simple roots};
  \item $P^{\vee} := \Hom (P,\ZM)$, called the \emph{dual weight lattice} and $\la , \ra : P^{\vee} \times P \to \ZM$ the natural pairing;
  \item $\Pi^\vee = \{ \alpha_i^{\vee} \mid i \in I \} \subset P^{\vee}$, called the set of \emph{simple coroots};
\end{enumerate}
satisfying the following properties:
\begin{enumerate}
  \item $\langle\alpha_i^{\vee} , \alpha_j \rangle = a_{ij}$ for all $i,j\in I$,
  \item $\Pi$ is linearly independent,
  \item $\forall i \in I$, $\exists \Lambda_i \in P$ such that $\langle\alpha_j^{\vee} , \Lambda_i \rangle = \delta_{ij}$ for all $j \in I$.
\end{enumerate}

Those $\Lambda_i$ are called the \emph{fundamental weights}. We denote by
$$P^+ := \{ \l \in P \mid \la \alpha_i^\vee , \l \ra \in \ZM_{\geqslant 0}\text{ for all } i\in I \}$$
the set of \emph{dominant integral weights}.
The free abelian group $Q := \oplus_{i\in I}\ZM \alpha_i$ is called the \emph{root lattice}.
Set $Q^+ = \sum_{i\in I} \ZM_{\geqslant 0}\alpha_i$. For $\beta = \sum_{i\in I} k_i \alpha_i \in Q^+$, we define the \emph{height} of $\beta$ to be $|\beta| = \sum_{i\in I} k_i$. For each $n\in\mathbb{N}$, we set $Q_n^{+} := \{ \beta \in Q^+ \mid |\beta| = n\}$.
Let $\hg = \QM \otimes_{\ZM}P^\vee$. Since $A$ is symmetrizable, there is a symmetric bilinear form $(,)$ on $\hg^*$ satisfying
\begin{align*}
  & (\alpha_i, \alpha_j) = d_i a_{ij} \quad (i,j \in I) \quad \text{and} \\
  & \la \alpha_i^\vee , \l \ra = \frac{2(\alpha_i,\l)}{(\alpha_i,\alpha_i)} \quad \text{for any }\l \in \hg^* \text{ and } i\in I.
\end{align*}

We can now recall the construction of Khovanov-Lauda-Rouquier algebra (or simply KLR algebra) $\RSS_\beta$ associated with a Cartan datum $(A,P,\Pi,P^\vee ,\Pi^{\vee})$ and $\beta\in Q_n^+$.
Let $\kb = \oplus_{n\in \NM}\kb_n$ be a graded commutative noetherian unital ring such that $\kb_0 = \km$ is a field. We fix a matrix $(Q_{i,j})_{i,j\in I}$ in $\kb [u,v]$ such that
\begin{align*}
  & Q_{i,j}(u,v) = Q_{j,i}(v,u), \\
  & Q_{i,i}(u,v) = 0, \\
  & Q_{i,j}(u,v)=\sum_{p,q\geqslant 0} c_{i,j,p,q}\,u^pv^q\ \text{if}\ i\neq j.
\end{align*}
where $c_{i,j,p,q}\in\kb_{-2(\alpha_i,\alpha_j)-(\alpha_i,\alpha_i)p - (\alpha_j,\alpha_j)q}$ and $c_{i,j,-a_{ij},0}\in \kb_0^\times$.

We denote by $\SG_n = \la s_1 , \ldots , s_{n-1}\ra$ the symmetric group on $n$ letters, where $s_i = (i,i+1)$ is the transposition on $i,i+1$.
Then $\SG_n$ acts naturally on $I^n$ by:
$$w.\nu := (\nu_{w^{-1}(1)},\ldots , \nu_{w^{-1}(n)}),$$
where $\nu = (\nu_1 , \ldots, \nu_n) \in I^n$.
The orbits of this action is identified with element of $Q_n^{+}$ and we denote by $I^\beta$ the corresponding orbit.
For any $k,m\in\mathbb{N}$ with $k\leq m$, we set $[k,m]:=\{k,k+1,\cdots,m\}$.

\begin{dfn}\label{KLR}
  The Khovanov-Lauda-Rouquier (KLR) algebra $\RSS_\beta$ associated with a Cartan datum $(A,P,\Pi,P^{\vee},\Pi^{\vee})$, $(Q_{i,j})_{i,j\in I}$ and $\beta \in Q_n^+$ is the associative algebra over $\kb$ generated by $e(\nu)\, (\nu\in I^\beta)$, $x_k\, (1\leqslant k \leqslant n)$, $\tau_l \, (1\leqslant l \leqslant n-1)$ satisfying the following defining relations:
  \begin{equation*}
    \begin{aligned}
      & e(\nu) e(\nu') = \delta_{\nu, \nu'} e(\nu), \ \
      \sum_{\nu \in I^{\beta}}  e(\nu) = 1, \\
      & x_{k} x_{l} = x_{l} x_{k}, \ \ x_{k} e(\nu) = e(\nu) x_{k}, \\
      & \tau_{l} e(\nu) = e(s_{l}(\nu)) \tau_{l}, \ \ \tau_{k} \tau_{l} = \tau_{l} \tau_{k} \ \ \text{if} \ |k-l|>1, \\
      & \tau_{k}^2 e(\nu) = Q_{\nu_{k}, \nu_{k+1}} (x_{k}, x_{k+1})e(\nu), \\
      & (\tau_{k} x_{l} - x_{s_k(l)} \tau_{k}) e(\nu) = \begin{cases}
      -e(\nu) \ \ & \text{if} \ l=k, \nu_{k} = \nu_{k+1}, \\
      e(\nu) \ \ & \text{if} \ l=k+1, \nu_{k}=\nu_{k+1}, \\
      0 \ \ & \text{otherwise},
      \end{cases} \\[.5ex]
      & (\tau_{k+1} \tau_{k} \tau_{k+1}-\tau_{k} \tau_{k+1} \tau_{k}) e(\nu)\\
      &\hspace*{8ex} =\begin{cases} \dfrac{Q_{\nu_{k}, \nu_{k+1}}(x_{k},
      x_{k+1}) - Q_{\nu_{k+2}, \nu_{k+1}}(x_{k+2}, x_{k+1})}{x_{k} - x_{k+2}}e(\nu) \ \ & \text{if} \
      \nu_{k} = \nu_{k+2}, \\
      0 \ \ & \text{otherwise}.
      \end{cases}
    \end{aligned}
  \end{equation*}
\end{dfn}

In particular, $\RSS_0 \cong \kb$, and $\RSS_{\alpha_i}$ is isomorphic to $\kb [x_1]$.
We will write $e(\beta -\alpha_i , i) = \operatornamewithlimits{\sum}\limits_{\substack{\nu\in I^\beta \\ \nu_n = i}}e(\nu)$ and more generally, for $\beta\geq\beta'\in Q_{n-t+1}^+$, and $\nu' \in I^{\beta'}$, $$e(\beta -\beta' , \nu'): = \operatornamewithlimits{\sum}\limits_{\substack{\nu\in I^\beta \\ (\nu_t , \ldots ,\nu_n) = \nu'}}e(\nu)$$  We also abbreviate $e(\ubr{i\cdots i}{k\,\text{copies}})$ as $e(i^k)$.

The algebra $\RSS_\beta$ is $\ZM$-graded whose grading is given by
\begin{equation*}
  \deg e(\nu) = 0, \qquad \deg x_k e(\nu) = (\alpha_{\nu_k},\alpha_{\nu_k}), \qquad \deg \tau_l e(\nu) = - (\alpha_{\nu_l},\alpha_{\nu_{l+1}}).
\end{equation*}

For any $f\in \kb [x_1,\ldots ,x_n]$, define $\del_k (f) = \frac{f - s_k(f)}{x_{k+1} - x_k}$ and call $\del_k $ the $k$-th Demazure operator.

\begin{lem} For any $\nu \in I^\beta$ with $\nu_k = \nu_{k+1}$ and any $f\in \kb [x_1,\ldots ,x_n]$, we have \begin{equation}\label{eq:KK42}
  f\tau_ke(\nu) = (\del_k(f) + \tau_k f + \tau_kf\tau_k(x_k - x_{k+1}))e(\nu),
\end{equation}
\begin{equation}\label{eq:TwoSideDem}
  \tau_k f \tau_k e(\nu) = \tau_k\del_k (f)e(\nu),
\end{equation}
\end{lem}

\begin{proof} These follow from the defining relations in Definition \ref{KLR} and some induction on the degree of $f$.
\end{proof}

\begin{prop}[\cite{KL09,R08,R11}]\label{prop:MBKLR}
  $\RSS_\beta$ is a free $\kb$-module with basis
  $$S_\beta = \{ \tau_wx^{\un{a}}e(\nu)\mid \nu \in I^\beta,\, w\in \SG_n,\, \un{a} = (a_1,\ldots ,a_n)\in \NM^n,\, 1\leqslant i \leqslant n  \},$$
  where $\tau_{w} := \tau_{i_1}\cdots \tau_{i_r}$ with $w = s_{i_1}\cdots s_{i_r}$ a preferred choice of reduced decomposition of $w$ and $x^{\un{a}} := x_1^{a_1}x_2^{a_2}\cdots x_n^{a_n}$.
\end{prop}

This basis is called the \emph{homogeneous monomial basis} of $\RSS_\beta$.

The centers of KLR algebras are also well-known. Recall that $\SG_n$ acts naturally on the \emph{polynomial subalgebra} $P_\beta := \bigoplus_{\nu\in I^\beta} \kb [x_1,x_2,\ldots ,x_n]e(\nu)$ by $$
wf(x_1,\cdots,x_n):=f(x_{w1},\cdots,x_{wn}),\quad w(e(\nu)):=e(w.\nu) .
$$ We have

\begin{prop}[\cite{KL09,R08}]\label{prop:CenKLR}
  The center $Z(\RSS_\beta)$ of $\RSS_\beta$ consists of $\SG_n$-fixed points (i.e. symmetric elements on $x_1,\cdots,x_n,e(\bnu),\nu\in I^\beta$) of $P_\beta$:
  $$Z(\RSS_\beta) = \bigg( \bigoplus_{\nu\in I^\beta}\kb [x_1, \ldots , x_n]e(\nu) \bigg)^{\SG_n}.$$
\end{prop}


Let $\Lambda \in P^+$ be a dominant integral weight.
We now study the cyclotomic KLR algebra $\RSS_\beta^\Lambda$.
For $i\in I$, choose a monic polynomial of degree $\la \alpha_i^\vee , \Lambda \ra$:
\begin{equation*}
  a_i^\Lambda (u) = \sum_{k=0}^{\la \alpha_i^\vee ,\Lambda \ra} c_{i;k} u^{\la \alpha_i^\vee , \Lambda\ra -k}
\end{equation*}
with $c_{i;k} \in \kb_{k(\alpha_i,\alpha_i)}$ and $c_{i;0} = 1$.

For $\beta \in Q_n^+$ and $1\leqslant k \leqslant n$, we define
\begin{equation*}
  a^\Lambda_\beta (x_k) = \sum_{\nu\in I^\beta} a^\Lambda_{\nu_k}(x_k)e(\nu) \in \RSS_\beta.
\end{equation*}
Hence $a^\Lambda (x_k)e(\nu)$ is a homogeneous element of $\RSS_\beta$ with degree $2(\alpha_{\nu_k},\Lambda)$.

\begin{dfn}\label{cyclotomicKLR}
  Set $I_{\Lambda,\beta} = \RSS_\beta a^\Lambda_\beta (x_1) \RSS_\beta$. The \emph{cyclotomic KLR algebra} $\RSS^\Lambda_\beta$ is defined to be the quotient algebra
  $$\RSS_\beta^\Lambda = \RSS_\beta/I_{\Lambda,\beta}.$$
\end{dfn}

The following is proved in \cite[Corollary 4.1, Theorem 4.5]{KK12}.

\begin{prop}
  The $\kb$-algebra $\RSS_\beta^\Lambda$ is a free $\kb$-module of finite rank.
\end{prop}

A natural question arose: how to find out a subset $S_\beta^\Lambda$ of $S_\beta$ in Proposition~\ref{prop:MBKLR}, so that the canonical image of $S_\beta^\Lambda$ in $\RSS_\beta^\Lambda$ gives rise to a basis of $\RSS_\beta^\Lambda$. We will answer this question in Section~\ref{sec:MBCKLR} when $\beta=\sum_{j=1}^n\alpha_{i_j}$ with $\alpha_{i_1},\cdots,\alpha_{i_n}$ pairwise distinct.
Another unsolved open problems on $\R[\beta]$ is the following center conjecture which reveals the relation between the centers of KLR algebras and their cyclotomic quotients.

\begin{cconj}\label{conj:CC} Let $\beta\in Q_n^+$ and $\Lam\in P^+$. Let $p_\beta^\Lam: \RSS_\beta \twoheadrightarrow \RSS_\beta^\Lambda$ be the canonical surjection. Then the induced map
  $$p_\beta^{\Lam}|_{Z(\RSS_\beta)}: Z(\RSS_\beta) \to Z(\RSS_\beta^\Lambda)$$
  is always surjective. In particular, the center $Z(\RSS_\beta^\Lambda)$ consists of symmetric elements in its KLR $x$ and $e(\nu)$ generators.
\end{cconj}

In Section~\ref{sec:DC} we will give an approach to this conjecture and justify our approach in Section~\ref{sec:MBCocen} when $\beta=\sum_{j=1}^n\alpha_{i_j}$ with $\alpha_{i_1},\cdots,\alpha_{i_n}$ pairwise distinct.
\bigskip

\section{The center conjecture and distinguished central elements}\label{sec:DC}

In this section, we assume that $\kb = \kb_0$ is a field. Fix $\Lambda \in P^+$ and $i \in I$. For any $j\in I$, $a_j^\Lambda (u) = u^{\la \alpha_j^{\vee} ,\Lambda \ra}$ and $a_j^{\Lam+\Lam_i}(u) = u^{\la \alpha_j^{\vee} ,\Lambda+\Lam_i \ra}$. In particular, $a_j^\Lambda (u) | a_j^{\Lam+\Lam_i}(u)$.
Therefore, $I_{\Lam+\Lam_i,\beta} \subset I_{\Lambda ,\beta}$, we have canonical surjection
$$ p_{\beta}^{\Lambda,i} : \RSS_\beta^{\Lam+\Lam_i} \twoheadrightarrow \RSS_\beta^\Lambda .$$
Furthermore, for any $\Lambda , \Lambda'\in P^+$, we have a canonical surjection $p_\beta^{\Lambda ,\Lambda'} : \RSS_\beta^{\Lambda + \Lambda'} \twoheadrightarrow \RSS_\beta^\Lambda$.
When $\beta,\Lambda,\Lambda'$ are clear in the context, we write $p_\beta$ instead of $p_\beta^{\Lambda ,\Lambda'}$ for simplicity.
These induce canonical homomorphisms between centers $$\ov{p}_\beta^{\Lambda ,\Lambda'}:=p_\beta^{\Lam,\Lam'}|_{Z(\RR_\beta^{\widetilde{\Lam}})} : Z(\RSS_\beta^{\widetilde{\Lambda}}) \to Z(\RSS_\beta^{\Lambda}).$$


\begin{thm}\label{thm:CC} Let $\beta \in Q_n^+$. Then the center conjecture~\ref{conj:CC} holds for any $\Lambda \in P^+$ if and only if for any $\Lambda \in P^+$ and any $i\in I$, $\ov{p}_\beta^{\Lambda,i}$ is surjective.
\end{thm}

\begin{proof} If the center conjecture~\ref{conj:CC} holds, then clearly for any $\Lambda \in P^+$ and any $i\in I$, $\ov{p}_\beta^{\Lambda,i}$ is surjective.

Conversely, assume that for any $\Lambda \in P^+$ and any $i\in I$, $\ov{p}_\beta^{\Lambda, i}$ is surjective. Then it immediately implies that for any $\Lambda,\Lambda' \in P^+$,
$\ov{p}_\beta^{\Lambda , \Lambda'}$ is surjective.

  Now, we fix $\Lambda \in P^+$ and let $z^\Lam\in Z(\R[\beta])$ be a homogeneous element of degree $d_z$. Note that by Proposition \ref{prop:MBKLR}, every homogeneous element in $\RSS_\beta$ has degree $\geq -2m_0\ell(w_0)=-m_0n(n-1)$, where $w_0$ is the unique longest element in $\mathfrak{S}_n$, $m_0:=\max\{d_i|i\in I\}$. We can take $\Lambda'\in P^+$ with $\la \alpha_i^{\vee} , \Lambda' \ra \gg 0$ for any $i\in I$, such that $I_{\Lam+\Lam',\beta} \subset (\RSS_\beta)_{> d_z+d}$, where $$
d:=\max\{0,\deg\tau_le(\nu),\deg x_ke(\nu)\mid \nu \in I^\beta, \, 1\leqslant l < n ,\, 1\leqslant k \leqslant n \}. $$

Since $\ov{p}_\beta^{\Lambda , \Lambda'}$ is surjective, we can choose $z^{\Lam+\Lam'}\in Z(\RR_\beta^{\Lam+\Lam'})$ such that $p^{\Lam,\Lam'}_\beta(z^{\Lam+\Lam'})=z^\Lam$.
Let $z\in \RR_\beta$ be a homogeneous element of degree $d_z$ such that $p^{\Lam+\Lam'}_\beta(z)=z^{\Lam+\Lam'}$. Since $p_\beta^\Lam=p^{\Lam,\Lam'}_\beta\circ p^{\Lam+\Lam'}_\beta$, it follows that $p_\beta^{\Lam}(z)=z^\Lam$.
The assumption $z^{\Lam+\Lam'}\in Z(\RR_\beta^{\Lam+\Lam'})$ implies that for any $y \in \{ e(\nu),\tau_le(\nu),x_ke(\nu)\mid \nu \in I^\beta, \, 1\leqslant l < n ,\, 1\leqslant k \leqslant n \}$, $$
zy - yz\in I_{\Lam+\Lam',\beta} .
$$
Combining this with the assumption $I_{\Lam+\Lam',\beta} \subset (\RSS_\beta)_{> d_z+d}$, we can deduce that $$
zy - yz=0,\quad \forall\,y \in \{ e(\nu),\tau_le(\nu),x_ke(\nu)\mid \nu \in I^\beta, \, 1\leqslant l < n ,\, 1\leqslant k \leqslant n \}.
$$
In other words, $z\in Z(\RR_\beta)$. This completes the proof of the theorem.
\end{proof}

This theorem can also be reinterpreted in terms of cocenters since $\RSS_\beta^\Lambda$ is a $\kb$-symmetric algebra.
Here we only give the definition of symmetric algebras, for more properties, refer to \cite{R08} for more details.

\begin{dfn} Let $R$ be a commutative domain and $A$ an $R$-algebra which is finitely generated projective as $R$-module.
  An $R$-linear morphism $t : A \to R$ is called a symmetrizing form if the morphism
  $$\hat{t}: A \to \Hom_{R}(A,R), \qquad a \mapsto (a'\mapsto t(a'a))$$
  is an $(A,A)$-bimodule isomorphism. In this case, we call $(A,t)$ (or simply $A$) a symmetric algebra over $R$.
\end{dfn}
In particular, if  $t : A \to R$ is a symmetrizing form, then $t(aa') = t(a'a), \forall\,a,a'\in A$. The map
  $t$ induces a perfect pairing on $A$:
  $$A \times A \to R , \qquad (a,a') \mapsto t(aa'),$$
  and a perfect pairing
  $$Z(A) \times A/ [A,A] \to R, \qquad (a,a') \mapsto t(aa')$$
  whenever $R$ is a field.

Let $p : A\twoheadrightarrow B$ be a surjective homomorphism between two symmetric $\kb$-algebras $(A,t_A)$ and $(B,t_B)$. The dual of $p$ is an injective $(A,A)$-bimodule homomorphism $\iota : B \into A$. By restriction, the map $p$ induces a homomorphism $\ov{p} : Z(A) \to Z(B)$ between the two centers, which may not be surjective in general, Via the above perfect pairing the map $\iota$ induces a homomorphism $\ov{\iota} : B/ [B,B] \to A/ [A,A]$ of $(Z(A),Z(A))$-bimodules.

\begin{lem}\label{lem:injdual}
  Let $p: A \twoheadrightarrow B$ be a surjective algebraic homomorphism between two symmetric $\kb$-algebras $(A,t_A)$ and $(B,t_B)$. Let $\iota:=p^*: B\hookrightarrow A$ be the dual of $p$. Then $z:=\iota(1_B)\in A$ is the unique element in the center $Z(A)$ of $A$ which satisfies that $t_A (az) = t_B (p(a))$, $\forall\, a \in A$. Moreover, we have
  $$\Ann_A(z) = \Ker (p),\quad \iota(p(a))=az, \,\,\forall\,a\in A. $$
\end{lem}

\begin{proof} By the definition of dual map, we see that $z:=\iota(1_B)\in A$ satisfies that $t_A (az) = t_B (p(a))$, $\forall\, a \in A$.

For any $a,b\in A$, we have $$\begin{aligned}
t_A((az-za)b)&=t_A(azb)-t_A(zab)=t_A(baz)-t_A(abz)=t_A((ba-ab)z)\\
&=t_B(p(ba-ab))=t_B(p(a)p(b)-p(b)p(a))=0 .
\end{aligned}
$$
It follows that $za-az=0$ and hence $z=\iota(1_B)\in A$ is in the center $Z(A)$ of $A$. If $z'\in Z(A)$ also satisfies that $t_A (az') = t_B (p(a))$, $\forall\, a \in A$. Then $t_A(a(z-z'))=t_B (p(a))-t_B (p(a))=0, \forall\,a\in A$, which implies that $z=z'$.

Let $a \in\Ann_A (z)$. For any $b \in B$, we can write $b = p(a')$, where $a' \in A$ as $p$ is surjective. Thus $$t_B (p(a)b) = t_B (p(aa')) = t_A (zaa') = 0.$$
  Since $t_B$ is non-degenerate and $b$ is arbitrary, we can deduce that $p(a) = 0$, i.e. $a\in \Ker (p)$.

  Conversely, if $a \in\Ker (p)$, for any $a' \in A$, we have
  $$t_A (a'az) = t_B (p(a'a)) = 0.$$
  The non-degeneracy of $t_A$ yields that $az = 0$, i.e. $a \in\Ann_A (z)$. We are done.
\end{proof}

We now turn to study the morphism $\bar{\iota}_\beta^{\Lambda,i}$.


\begin{dfn} Let $\Lam\in P^+$ and $\beta\in Q_n^+$. We define $$
d_{\Lambda ,\beta}:=2(\Lam,\beta)-(\beta,\beta) .
$$
\end{dfn}

Let $\beta\in Q_n^+$. For each $\nu\in I^\beta$, we define $$
r_\nu:=\prod_{k<l}r_{\nu_k,\nu_l},\quad \text{where}\quad r_{i,j}=\begin{cases} c_{i,j,-a_{ij},0}, &\text{if $i\neq j$;}\\
0, &\text{if $i=j$,}\end{cases}
$$
and set \begin{equation}
r(\beta,\nu_n):=\prod_{k=1}^{n-1}r_{\nu_k,\nu_n} .
\end{equation}

Let $\Lam\in P^+$ and $\nu=(\nu_1,\cdots,\nu_n)\in I^\beta$. For each $1\leq k\leq n$, let $$
\veps_{k,\nu_k}: e(\nu_1,\cdots,\nu_k)\R[\sum_{j=1}^k\alpha_{\nu_j}]e(\nu_1,\cdots,\nu_k)\rightarrow e(\nu_1,\cdots,\nu_{k-1})\R[\sum_{j=1}^{k-1}\alpha_{\nu_j}]e(\nu_1,\cdots,\nu_{k-1})
$$
be the homomorphism $\veps_{\nu_k}$ as defined in \cite[A.3]{SVV}, which is the same as the restriction of $\veps'_{\nu_k,\Lam-\sum_{i=1}^{k-1}\alpha_{\nu_i}}$ (in the notation of \cite[Theorem 3.8]{SVV}) to the bi-weight subspace
$e(\nu_1,\cdots,\nu_k)\R[\sum_{j=1}^k\alpha_{\nu_j}]e(\nu_1,\cdots,\nu_k)$. Define a map $t_{\Lambda ,\beta}$ as follows: for any $\nu,\nu'\in I^\beta$ and $x\in\R[\beta]$, $$
t_{\Lambda ,\beta}(e(\nu)xe(\nu')):=\begin{cases} r_{\nu}\veps_{1,\nu_1}\circ\cdots\circ\veps_{n,\nu_n}(e(\nu)xe(\nu')), &\text{if $\nu=\nu'$;}\\
0, &\text{if $\nu\neq\nu'$.}\end{cases}
$$
We extend the above map $\kb$-linearly to a $\kb$-linear map $t_{\Lambda ,\beta}: \R[\beta]\rightarrow\kb$.

\begin{prop}[\cite{SVV}] Let $\Lam\in P^+$ and $\beta\in Q_n^+$. The map $t_{\Lambda ,\beta}$ is a homogeneous symmetrizing form with degree $-d_{\Lambda ,\beta}$ on
$\RSS_\beta^\Lambda$. In particular, $\RSS_\beta^\Lambda$ is a symmetric $\kb$-algebra.
\end{prop}

For each $\Lam,\Lam'\in P^+$. We denote by $\ov{\iota}_\beta^{\Lambda ,\Lambda'}: \RR_\beta^{\Lam}/[\R[\beta],\R[\beta]]\rightarrow
\RR_\beta^{\Lam+\Lam'}/[\RR_\beta^{\Lam+\Lam'},\RR_\beta^{\Lam+\Lam'}]$ the dual of the map $\ov{p}_\beta^{\Lambda , \Lambda'}: Z(\RR_\beta^{\Lam+\Lam'})\rightarrow Z(\R[\beta])$. Similarly, we use $$
\ov{\iota}_\beta^{\Lambda,i}: \RR_\beta^{\Lam}/[\R[\beta],\R[\beta]]\rightarrow
\RR_\beta^{\Lam+\Lam_i}/[\RR_\beta^{\Lam+\Lam_i},\RR_\beta^{\Lam+\Lam_i}]$$
to denote the dual of the map $\ov{p}_\beta^{\Lam,i}: Z(\RR_\beta^{\Lam+\Lam_i})\rightarrow
Z(\R[\beta])$.

\begin{thm}\label{thm:CC2}
  Let $\beta \in Q_n^+$. Then the Center Conjecture~\ref{conj:CC} holds for all $\Lambda \in P^+$ if and only if $\ov{\iota}_\beta^{\Lambda,i}$ is injective for any $\Lambda \in P^+$ and $i\in I$.
\end{thm}

\begin{proof} This is a consequence of Theorem \ref{thm:CC}.
\end{proof}

\begin{prop}[\cite{SVV}]\label{prop:IndBubb}
  For each $\beta \in Q_n^+$ and $\nu\in I^\beta$, we have
  $$B_{\pm i, \lambda}(z)e(\nu) = z^{\mp t_i}a_i^{\Lambda}(z^{-1})^{\mp 1}\prod_{k=1}^{n}q_{i\nu_k}(z, x_k)^{\mp 1}e(\nu).$$
  where $B_{\pm i, \lambda}(z) = \sum_{k\in \ZM}B_{\pm i, \lambda}^k z^k\in\R[\beta][[z]]$ and $t_i = \la \alpha_i^{\vee} , \Lambda\ra$.
\end{prop}

The following lemma plays a crucial in the proof of the main result in this section.

\begin{lem}\label{keyDifflem} Let $\beta \in Q_n^+$, $\nu \in I^\beta$, $i,j \in I$. Set $\widetilde{\Lam}:=\Lam+\Lam_i, \lam=\Lam-\beta, \widetilde{\lam}:=\widetilde{\Lam}-\beta$. Let $a \in e(\nu , j)\RSS_{\beta +\alpha_j}^{\widetilde{\Lambda}}e(\nu,j)$. Then \begin{enumerate}
\item if $i\neq j$, then $p_\beta^{\Lam,i}(\hat\varepsilon_{j,\widetilde{\lam}}(a))=\hat\varepsilon_{j,\lam}(p_{\beta+\alpha_j}^{\Lam,i}(a))$;
\item if $i=j$, then $p_\beta^{\Lam,i}(\hat\varepsilon_{j,\widetilde{\lam}}(ax_{n+1}))=\hat\varepsilon_{j,\lam}(p_{\beta+\alpha_j}^{\Lam,i}(a))$.
\end{enumerate}
\end{lem}

\begin{proof} 1) Set $\lam_j:=\langle h_j,\lam\rangle,\, \widetilde{\lambda}_j := \langle h_j,\widetilde{\lambda}\rangle$. If $\lam_j>0$, then by \cite[(6)]{SVV} we have a decomposition: \begin{equation}\label{decomp1}
a = e(\nu,j)\mu_{\tau_n}(\pi(a))e(\nu,j) + \sum_{k=0}^{\widetilde{\lambda}_j-1}e(\nu,j)p_k(a)x_{n+1}^ke(\nu,j) ,
\end{equation}
where $\pi(a)\in \RR_\beta^{\widetilde{\Lam}}e(\beta-\alpha_j,j)\otimes_{\RR_{\beta-\alpha_j}^{\widetilde{\Lam}}}e(\beta-\alpha_j,j)\RR_\beta^{\widetilde{\Lam}}$, $p_k(a)\in \R[\beta]$, $k=0,1,\cdots,\widetilde{\lambda}_j-1$, are unique elements such that (\ref{decomp1}) holds.

Note that the assumption $i\neq j$ implies that $\lam_j=\widetilde{\lam}_j$. We can  apply $p_{\beta+\alpha_j}^{\Lam,i}$ to (\ref{decomp1}) to get the decomposition (\cite[(6)]{SVV}) of $p_{\beta+\alpha_j}^{\Lam,i}(a)$ in $\R[\beta]$. By the uniqueness related to the decomposition we can deduce that $(p_{\beta}^{\Lam,i}\otimes p_{\beta}^{\Lam,i})(\pi(a))=\pi(p_{\beta+\alpha_j}^{\Lam,i}(a))$ and $p_{\beta}^{\Lam,i}(p_k(a))=p_k(p_{\beta+\alpha_j}^{\Lam,i}(a))$, $k=0,1,\cdots,\widetilde{\lambda}_j-1=\lam_j-1$.
In particular, we have $$
p_\beta^{\Lam,i}(\hat\varepsilon_{j,\widetilde{\lam}}(a))=p_\beta^{\Lam,i}(p_{\lambda_j-1}(a)) = p_{\lambda_j-1}(p_{\beta+\alpha_j}^{\Lam,i}(a))=\hat\varepsilon_{j,\lam}(p_{\beta+\alpha_j}^{\Lam,i}(a)).
$$

If $\lambda_j\leqslant 0$, then we have $\hat\varepsilon_{j,\widetilde{\lam}}(a)=\mu_{x_n^{-\widetilde{\lam}_j}}(\widetilde{a})=\mu_{x_n^{-\lam_j}}(\widetilde{a})$, where $\widetilde{a}\in \RR_\beta^{\widetilde{\Lam}}e(\beta-\alpha_j,j)\otimes_{\RR_{\beta-\alpha_j}^{\widetilde{\Lam}}}e(\beta-\alpha_j,j)\RR_\beta^{\widetilde{\Lam}}$ is the unique element such that $$
e(\nu,j)\mu_{\tau_n}(\widetilde{a})e(\nu,j)=a,\quad\, \mu_{x_n^k}(\widetilde{a})=0,\,\,\,k=0,1,\cdots,-\widetilde{\lambda}_j-1=-\lam_j-1.
$$
By the uniqueness we can deduce that $$
(p_{\beta}^{\Lam,i}\otimes p_{\beta}^{\Lam,i})(\widetilde{a})=\widetilde{p_{\beta+\alpha_j}^{\Lam,i}(a)} .
$$
It follows that $$
p_\beta^{\Lam,i}(\hat\varepsilon_{j,\widetilde{\lam}}(a))=p_\beta^{\Lam,i}(\mu_{x_n^{-\lam_j}}(\widetilde{a})) = \mu_{x_n^{-\lam_j}}(\widetilde{p_{\beta+\alpha_j}^{\Lam,i}(a)})=\hat\varepsilon_{j,\lam}(p_{\beta+\alpha_j}^{\Lam,i}(a)).
$$
This proves 1).

2) Suppose $i=j$. Then $\widetilde{\lam}_i=\la \alpha_i^\vee,\widetilde{\Lambda} \ra= \la \alpha_i^\vee,\Lambda \ra+1=\lam_i+1$.
   Our discussion will be divided into four cases:


\smallskip\noindent
{\it Case 1.}
$\lambda_i > 0$. By \cite[(6)]{SVV}, we have
          \begin{equation}\label{eq:case1}
            a= e(\nu,i)\mu_{\tau_{n}}\big(\pi(a)\big)e(\nu,i) + \sum_{k=0}^{\lambda_i} e(\nu,i)p_{k}(a)x_{n+1}^k e(\nu,i),
          \end{equation}
          where $\pi(a)\in \RSS_\beta^{\widetilde{\Lambda}}e(\beta -\alpha_i ,i)\otimes_{\RSS_{\beta-\alpha_i}^{\widetilde{\Lambda}}}e(\beta -\alpha_i ,i)\RSS_\beta^{\widetilde{\Lambda}}$, $p_k(a) \in \RSS_\beta^{\widetilde{\Lambda}}$, $k=0,1,\cdots,\lam_i$. By the relation $\tau_{n}x_{n+1} e(\beta - \alpha_i,i^2) = x_{n}\tau_{n}e(\beta - \alpha_i,i^2) + e(\beta - \alpha_i,i^2)$, we can deduce
          $$\hat\varepsilon_{j,\widetilde{\lam}}(ax_{n+1}) =\hat\varepsilon_{i,\widetilde{\lam}}(ax_{n+1}) = {p}_{\lambda_i}(ax_{n+1}) = {p}_{\lambda_i - 1}(a) + {p}_{\lambda_i}(a){p}_{\lambda_i}(x_{n+1}^{\lambda_i+1}e(\nu,i)).$$
          Applying $p_{\beta+\alpha_i}^{\Lam,i}$ to \eqref{eq:case1} and using the $\RSS_\beta^{\Lambda}$-bilinearity of $\veps_{i,\lam}$ (\cite[(6)]{SVV}), we can deduce that
          $$p_{\beta+\alpha_i}^{\Lam,i}(a) = \mu_{\tau_{n}}\bigl((p_{\beta}^{\Lam,i}\otimes p_{\beta}^{\Lam,i})\pi(a)\bigr) + \sum_{k=0}^{\lambda_i-1}p_{\beta}^{\Lam,i}(p_k(a))x_{n+1}^k e(\nu,i) + p_{\beta}^{\Lam,i}(p_{\lambda_i}(a))x_{n+1}^{\lambda_i}e(\nu,i),$$
          and therefore
          $$\hat\varepsilon_{j,\lam}(p_{\beta+\alpha_i}^{\Lam,i}(a)) = p_{\beta}^{\Lam,i}(p_{\lambda_i - 1}(a)) + p_{\beta}^{\Lam,i}(p_{\lambda_i}(a))p_{\lambda_i-1}(x_{n+1}^{\lambda_i}e(\nu,i)).$$
          So Part 2) of the lemma will follow from the following equality:
          \begin{equation}\label{eq:redcase1}
            p_{\beta}^{\Lam,i}\bigl(\hat\varepsilon_{i,\widetilde{\lam}}(x_{n+1}^{\lambda_i+1}e(\nu,i))\bigr) = \hat\varepsilon_{i,\lam}(x_{n+1}^{\lambda_i}e(\nu,i)),
                    \end{equation}
where the element $x_{n+1}^{\lam_i+1}e(\nu,i)$ on the left handside of the above equality is defined in $\RR_{\beta+\alpha_i}^{\widetilde{\Lam}}$, while the
element $x_{n}^{\lam_i}e(\nu,i)$ on the right handside of
the above equality is defined in $\R[\beta]$. The above equality follows from Proposition~\ref{prop:IndBubb}, the fact that $p_\beta^{\Lam,i}(B_{+i, \widetilde{\lambda}}^1) = B_{+i, \lambda}^1$ and the
equalities (\cite[Definition A.1]{SVV}) $$
B_{+i,\widetilde{\lam}}^1e(\nu)=\veps_{i,\widetilde{\lam}}\bigl(x_{n+1}^{\widetilde{\lambda}_i}e(\nu,i)\bigr)=\veps_{i,\widetilde{\lam}}(x_{n+1}^{\lambda_i+1}e(\nu,i)),\,\,
B_{+i,\lam}^1e(\nu)=\veps_{i,\lam}\bigl(x_{n+1}^{\lambda_i}e(\nu,i)\bigr),
$$ in this case.

\smallskip\noindent
{\it Case 2.}
$\lambda_i = 0$. By \cite[(6)]{SVV}, we have
          \begin{equation}\label{eq:case2}
            a= e(\nu,i)\mu_{\tau_{n}}\big(\pi(a)\big)e(\nu,i) +e(\nu,i)p_0(a)e(\nu,i),
          \end{equation}
where $\pi(a)\in \RSS_\beta^{\widetilde{\Lambda}}e(\beta -\alpha_i ,i)\otimes_{\RSS_{\beta-\alpha_i}^{\widetilde{\Lambda}}}e(\beta -\alpha_i ,i)\RSS_\beta^{\widetilde{\Lambda}}$, $p_0(a) \in \RSS_\beta^{\widetilde{\Lambda}}$.
          Again, since $$\tau_{n}x_{n+1} e(\beta - \alpha_i,i^2) = x_{n}\tau_{n}e(\beta - \alpha_i,i^2) + e(\beta - \alpha_i,i^2),$$ we can deduce that (because $\lam_i=0$)
          $$\hat\varepsilon_{i,\widetilde{\lam}}(ax_{n+1}) = \mu_1\big(\pi(a)\big) + p_{0}(a)p_{0}(x_{n+1}e(\nu,i)).$$
          Applying $p_{\beta+\alpha_i}^{\Lam,i}$ to \eqref{eq:case2} and using the $R^{\Lambda}(\beta)$-bilinearity of the map $z\mapsto \pi(z)$ in \cite[(6)]{SVV}, we get in $\R[\beta+\alpha_i]$ that
          $$\begin{aligned}
          p_{\beta+\alpha_i}^{\Lam,i}(a)& = \mu_{\tau_{n}}\big((p_{\beta}^{\Lam,i}\otimes p_{\beta}^{\Lam,i})(\pi(a)) \big) + p_{\beta}^{\Lam,i}(p_{0}(a))e(\nu,i)\\
          &=\mu_{\tau_{n}}\big((p_{\beta}^{\Lam,i}\otimes p_{\beta}^{\Lam,i})(\pi(a)) \big) + p_{\beta}^{\Lam,i}(p_{0}(a))\mu_{\tau_{n}}\big( \pi(e(\nu,i)) \big),\end{aligned}
          $$
          and therefore by applying \cite[Theorem 3.8]{SVV} we get that
          $$\hat\varepsilon_{i,\lam}(p_{\beta+\alpha_i}^{\Lam,i}(a)) = p_{\beta}^{\Lam,i}(\mu_1\big( \pi(a) \big)) + p_{\beta}^{\Lam,i}(p_{0}(a))\mu_{1}\big( \pi(e(\nu,i)) \big).$$
          So Part 2) of the lemma will follow from the following equality:
          \begin{equation}\label{eq:redcase2}
           p_{\beta}^{\Lam,i}(\hat\varepsilon_{i,\widetilde{\lam}}(x_{n+1}e(\nu, i))) = \hat\varepsilon_{i,\lam}(e(\nu, i));
          \end{equation}
where the element $x_{n+1}e(\nu,i)$ on the left handside of the above equality is defined in $\RR_{\beta+\alpha_i}^{\widetilde{\Lam}}$, while the element
$x_{n}e(\nu,i)$ on the right handside of the above equality is defined in $\R[\beta]$. The above equality follows from Proposition~\ref{prop:IndBubb}, the fact
that $p_\beta^{\Lam,i}(B_{+i, \widetilde{\lambda}}^1) = B_{+i, \lambda}^1$ and the equalities (\cite[Definition A.1]{SVV}) $$
B_{+i,\widetilde{\lam}}^1e(\nu)=\veps_{i,\widetilde{\lam}}\bigl(x_{n+1}e(\nu,i)\bigr),\,\, B_{+i,\lam}^1e(\nu)=\veps_{i,\lam}\bigl(e(\nu,i)\bigr)
$$ in this case.

\smallskip\noindent
{\it Case 3.} $\lambda_i < -1$. By \cite[(8)]{SVV}, we have
          \begin{equation}\label{eq:case3}
            a = e(\nu,i)\mu_{\tau_{n}}\big( \widetilde{a} \big)e(\nu,i),
          \end{equation}
where $\widetilde{a}\in \RR_\beta^{\widetilde{\Lam}}e(\beta-\alpha_i,i)\otimes_{\RR_{\beta-\alpha_i}^{\widetilde{\Lam}}}e(\beta-\alpha_i,i)\RR_\beta^{\widetilde{\Lam}}$ is the unique element such that (\ref{eq:case3}) holds and $$
\mu_{x_n^k}(\widetilde{a})=0,\,\,\,\,\,k=0,1,\cdots,-\widetilde{\lambda}_i-1=-\lam_i-2.
$$
          By the relation $\tau_{n}x_{n+1} e(\beta - \alpha_i,i^2) = x_{n}\tau_{n}e(\beta - \alpha_i,i^2) + e(\beta - \alpha_i,i^2)$, we can get that
          \begin{equation*}
                        ax_{n+1}  = \mu_{x_{n}\tau_{n}}(\widetilde{a}) + \mu_1(\widetilde{a})= \mu_{x_{n}\tau_{n}}(\widetilde{a}).
                                    \end{equation*}
where $\mu_1(\widetilde{a})=0$ follows from \cite[(8)]{SVV} and the assumption $\lam_i<-1$.

Recall the definition of the elements $\{\widetilde{\pi}_\ell|0\leq \ell\leq -\widetilde{\lambda}_i-1=-\lam_i-2\}$ in \cite[(8)]{SVV}.
In particular, $\mu_{\tau_{n}}(\widetilde{\pi}_{-\lambda_i-2})=0$. Therefore, \begin{equation}\label{axnplus}
 ax_{n+1}=\mu_{x_{n}\tau_{n}}(\widetilde{a})= \mu_{\tau_{n}}\Bigl( b - \hat\varepsilon_{i,\widetilde{\lam}}(a)\widetilde{\pi}_{-\lambda_i-2} \Bigr),
\end{equation}
where $b:= \sum_z z_1x_{n}\otimes z_2$ and $\widetilde{a} = \sum_z z_1\otimes z_2$. Note that the element $b - \hat\varepsilon_{i,\widetilde{\lam}}(a)\widetilde{\pi}_{-\lambda_i-2}$ satisfies that $$
\mu_{x_n^k}\bigl(b - \hat\varepsilon_{i,\widetilde{\lam}}(a)\widetilde{\pi}_{-\lambda_i-2}\bigr)=0,\,\,\,k=0,1,\cdots,-\widetilde{\lam}_i-1=-\lam_i-2.
$$
So by the uniqueness in the definition of $\widetilde{ax_{n+1}}$ (\cite[(8)]{SVV}) we can deduce that $\widetilde{ax_{n+1}}=b - \hat\varepsilon_{i,\widetilde{\lam}}(a)\widetilde{\pi}_{-\lambda_i-2}$.
Therefore, \begin{equation}\label{left}\hat\varepsilon_{i,\widetilde{\lam}}(ax_{n+1}) = \mu_{x_{n}^{-\lambda_i-1}}\big( \widetilde{ax_{n+1}} \big) = \mu_{x_{n}^{-\lambda_i}}(\widetilde{a}) - \hat\varepsilon_{i,\widetilde{\lam}}(a)\mu_{x_{n}^{-\lambda_i-1}}\big(\widetilde{\pi}_{-\lambda_i-2}\big).\end{equation}
          By definition, we have that \begin{equation}\label{muqq}
          \mu_{x_n^{-\lam_i-1}}\bigl((p_{\beta}^{\Lam,i}\otimes p_{\beta}^{\Lam,i})\widetilde{a}\bigr)=p_{\beta}^{\Lam,i}\bigl(\veps_{i,\widetilde{\lam}}(a)\bigr) .
          \end{equation}
We claim that \begin{equation}\label{wtildepa3}\widetilde{p_{\beta+\alpha_i}^{\Lam,i}(a)} =(p_{\beta}^{\Lam,i}\otimes p_{\beta}^{\Lam,i})(\widetilde{a}) - p_{\beta}^{\Lam,i}(\hat\varepsilon_{i,\widetilde{\lam}}(a))\widetilde{\pi}_{-\lambda_i-1}.\end{equation}
In fact, since $$\begin{aligned}
&\quad\,\mu_{\tau_n}\bigl((p_{\beta}^{\Lam,i}\otimes p_{\beta}^{\Lam,i})(\widetilde{a})\bigr)-\mu_{\tau_n}\bigl(p_{\beta}^{\Lam,i}(\hat\varepsilon_{i,\widetilde{\lam}}(a))\widetilde{\pi}_{-\lambda_i-1}\bigr)\\
&=p_{\beta+\alpha_i}^{\Lam,i}(a)-p_{\beta}^{\Lam,i}(\hat\varepsilon_{i,\widetilde{\lam}}(a))\mu_{\tau_n}\bigl(\widetilde{\pi}_{-\lambda_i-1}\bigr)\\
&=p_{\beta+\alpha_i}^{\Lam,i}(a)-0=p_{\beta+\alpha_i}^{\Lam,i}(a),\end{aligned}
$$
and $$\begin{aligned}
&\quad\,\mu_{x_n^{-\lam_i-1}}\bigl((p_{\beta}^{\Lam,i}\otimes p_{\beta}^{\Lam,i})(\widetilde{a})\bigr)-\mu_{x_n^{-\lam_i-1}}\bigl(p_{\beta}^{\Lam,i}(\hat\varepsilon_{i,\widetilde{\lam}}(a))\widetilde{\pi}_{-\lambda_i-1}\bigr)\\
&=p_{\beta}^{\Lam,i}(\hat\varepsilon_{i,\widetilde{\lam}}(a))-p_{\beta}^{\Lam,i}(\hat\varepsilon_{i,\widetilde{\lam}}(a))=0,
\end{aligned}
$$
and for $k=0,1,\cdots,-\lam_i-2$, $$
\mu_{x_n^{k}}\bigl((p_{\beta}^{\Lam,i}\otimes p_{\beta}^{\Lam,i})(\widetilde{a})\bigr)-\mu_{x_n^{k}}\bigl(p_{\beta}^{\Lam,i}(\hat\varepsilon_{i,\widetilde{\lam}}(a))\widetilde{\pi}_{-\lambda_i-1}\bigr)=0-0=0,
$$
our claim follows from the uniqueness (\cite[(8)]{SVV}) in the definition of $\widetilde{p_{\beta+\alpha_i}^{\Lam,i}(a)}$.

As a result of (\ref{wtildepa3}), we see that $$
\hat\varepsilon_{i,\lam}(p_{\beta+\alpha_i}^{\Lam,i}(a)) = \mu_{x_{n}}\big( \widetilde{a} \big) - p_{\beta}^{\Lam,i} (\hat\varepsilon_{n+1,i}(a))\mu_{x_{n}^{-\lambda_i}}\big(\widetilde{\pi}_{-\lambda_i-1}\big). $$

Now comparing (\ref{left}) with the above equality, to prove the part 2) of the lemma, it suffices to show that
          \begin{equation}\label{eq:redcase3}
            p_{\beta}^{\Lam,i}\Bigl( \mu_{x_{n}^{-\lambda_i-1}}\big(\widetilde{\pi}_{-\lambda_i-2}\big) \Bigr) = \mu_{x_{n}^{-\lambda_i}}\big(\widetilde{\pi}_{-\lambda_i-1}\big).
          \end{equation}
where the element $\widetilde{\pi}_{-\lambda_i-2}$ on the left handside of the above equality is defined in $\RR_{\beta+\alpha_i}^{\widetilde{\Lam}}$, while the element
$\widetilde{\pi}_{-\lambda_i-1}$ on the right handside of the above equality is defined in $\R[\beta]$. The above equality follows from Proposition~\ref{prop:IndBubb}, the fact that $p_\beta^{\Lam,i}(B_{+i, \widetilde{\lambda}}^1) = B_{+i, \lambda}^1$ and the equalities (\cite[Definition A.1]{SVV} $$
B_{+i,\widetilde{\lam}}^1=-\mu_{x_{n}^{-\lambda_i-1}}\big(\widetilde{\pi}_{-\lambda_i-2}\big),\,\, B_{+i,\lam}^1=-\mu_{x_{n}^{-\lambda_i}}\big(\widetilde{\pi}_{-\lambda_i-1})
$$ in this case.

%

\smallskip\noindent
{\it Case 4.} $\lambda_i = -1$. By \cite[(8)]{SVV}, we have
          \begin{equation}\label{eq:case4}
            a = e(\nu,i)\mu_{\tau_{n}}\big( \widetilde{a} \big)e(\nu,i),
          \end{equation}
where $\widetilde{a}\in \RR_\beta^{\widetilde{\Lam}}e(\beta-\alpha_i,i)\otimes_{\RR_{\beta-\alpha_i}^{\widetilde{\Lam}}}e(\beta-\alpha_i,i)\RR_\beta^{\widetilde{\Lam}}$ is the unique element such that (\ref{eq:case3}) holds.
By the relation $\tau_{n}x_{n+1} e(\beta - \alpha_i,i^2) = x_{n}\tau_{n}e(\beta - \alpha_i,i^2) + e(\beta - \alpha_i,i^2)$, we can get that
          \begin{equation*}
                        ax_{n+1}  = \mu_{x_{n}\tau_{n}}(\widetilde{a}) + \mu_1(\widetilde{a}).
                                    \end{equation*}

Since $\lam_i=-1$ implies that $\widetilde{\lam}_i=0$, it follows that $\veps_{i,\widetilde{\lam}}(a)=\mu_1(\widetilde{a})$ (\cite[Theorem 3.8]{SVV}). Therefore, \begin{equation}\label{axnplus4}
 ax_{n+1}=\mu_{x_{n}\tau_{n}}(\widetilde{a})= \mu_{\tau_{n}}\Bigl( b + \hat\varepsilon_{i,\widetilde{\lam}}(a)\pi(e(\beta,i))\Bigr),
\end{equation}
where $b:= \sum_z z_1x_{n}\otimes z_2$ and $\widetilde{a} = \sum_z z_1\otimes z_2$, and we have used the equalities: for $e(\nu,i)\in\RR_{\beta+\alpha_i}^{\widetilde{\Lam}}$,
 $\pi(e(\beta,i))=\widetilde{e(\beta,i)}$ and  $\mu_{\tau_n}\pi(e(\beta,i))=e(\beta,i)$. So by the uniqueness in the definition of $\widetilde{ax_{n+1}}$ (\cite[(8)]{SVV}) we can deduce that $\widetilde{ax_{n+1}}=b + \hat\varepsilon_{i,\widetilde{\lam}}(a)\pi(e(\beta,i))$.
Therefore, \begin{equation}\label{left4}\hat\varepsilon_{i,\widetilde{\lam}}(ax_{n+1}) = \mu_{1}\big( \widetilde{ax_{n+1}} \big) = \mu_{x_{n}}(\widetilde{a}) + \hat\varepsilon_{i,\widetilde{\lam}}(a)\mu_{1}\big(\pi(e(\beta,i))\big).\end{equation}
          By definition, we have that \begin{equation}\label{muqq4}
          \mu_{1}\bigl((p_{\beta}^{\Lam,i}\otimes p_{\beta}^{\Lam,i})\widetilde{a}\bigr)=p_{\beta}^{\Lam,i}\bigl(\veps_{i,\widetilde{\lam}}(a)\bigr) .
          \end{equation}
We claim that \begin{equation}\label{wtildepa}\widetilde{p_{\beta+\alpha_i}^{\Lam,i}(a)} =(p_{\beta}^{\Lam,i}\otimes p_{\beta}^{\Lam,i})(\widetilde{a})-\mu_1(\widetilde{a})\widetilde{\pi}_0,\end{equation}
where the element $\widetilde{\pi}_0$ in the righthand side of above equality is defined in $\R[\beta]$.

In fact, since $\mu_{\tau_n}(\widetilde{\pi}_0)=0$ and $\mu_{1}(\widetilde{\pi}_0)=1$, our claim follows from the uniqueness (\cite[(8)]{SVV}) in the definition of $\widetilde{p_{\beta+\alpha_i}^{\Lam,i}(a)}$ and the following two equalities: $$
\mu_{\tau_n}\bigl((p_{\beta}^{\Lam,i}\otimes p_{\beta}^{\Lam,i})(\widetilde{a})\bigr)
=p_{\beta+\alpha_i}^{\Lam,i}(a),\,\,\, \mu_{1}\bigl((p_{\beta}^{\Lam,i}\otimes p_{\beta}^{\Lam,i})(\widetilde{a})\bigr)=0.
$$

As a result of (\ref{wtildepa}), we see that $$
\hat\varepsilon_{i,\lam}(p_{\beta+\alpha_i}^{\Lam,i}(a)) = p_{\beta}^{\Lam,i} \Bigl( \mu_{x_{n}}\big( \widetilde{a} \big) \Bigr)-\mu_1(\widetilde{a})\mu_{x_n}(\widetilde{\pi}_0). $$

Now comparing (\ref{left4}) with the above equality, to prove the part 2) of the lemma, it suffices to show that
          \begin{equation}\label{eq:redcase4}
            p_{\beta}^{\Lam,i}\Bigl( \mu_{1}\big(\pi(e(\beta,i)\big) \Bigr) = -\mu_{x_{n}}(\widetilde{\pi}_0).
          \end{equation}
where the element $\pi(e(\nu,i)$ on the left handside of the above equality is defined in $\RR_{\beta+\alpha_i}^{\widetilde{\Lam}}$, while the element
$\widetilde{\pi}_0$ on the right handside of the above equality is defined in $\R[\beta]$. The above equality follows from Proposition~\ref{prop:IndBubb}, the fact that $p_\beta^{\Lam,i}(B_{+i, \widetilde{\lambda}}^1) = B_{+i, \lambda}^1$ and the equalities (\cite[Definition A.1]{SVV} $$\begin{aligned}
B_{+i,\widetilde{\lam}}^1&= \mu_1(\widetilde{e(\beta,i)})=\mu_{1}\big(\pi(e(\beta,i)\big),\\
B_{+i,\lam}^1&=-\mu_{x_{n}}(\widetilde{\pi}_0)
\end{aligned}
$$ in this case.
This completes the proof of the lemma.
\end{proof}

\begin{dfn} Let $\beta\in Q_n^+$. For any $i\in I$, we define
\begin{equation}\label{eq:DC}
  z(i,\beta) = \sum_{\nu \in I^\beta} \prod_{\substack{1\leq k\leq n\\ \nu_k = i}} x_ke(\nu) \in \RSS_\beta .
\end{equation}
For any $\Lam'=\sum_{j\in I}n_j\Lam_j\in P^+$, we define \begin{equation}\label{eq:DC2}
z(\Lambda' , \beta) = \prod_{i\in I} z(i,\beta)^{n_i} \in \RSS_\beta ,
\end{equation}
\end{dfn}
In particular, $z(\Lambda_i,\beta) = z(i,\beta)$. The element $z(i,\beta)$ is a symmetric elements in the KLR generators $x_1,\cdots,x_n,e(\nu), \nu\in I^\beta$. It follows from Proposition~\ref{prop:CenKLR} that $z(i,\beta)$ and hence $z(\Lam',\beta)$ belong to $Z(\RSS_\beta)$.

Henceforth, by some abuse of notations, we shall often use the same notations $z(i,\beta), z(\Lam',\beta)$ to denote their images in the cyclotomic KLR algebra $Z(\RSS_\beta^{\widetilde{\Lambda}})$.
The following theorem is the main result of this section.

\begin{thm}\label{thm:annker} Let $\Lam\in P^+, \beta\in Q_n^+$ and $i\in I$. Then $$
\iota_\beta^{\Lambda,i}(1_{\R[\beta]})=z(i,\beta)\in\RR_\beta^{\Lam+\Lam_i}.
$$
In particular, for any $a\in\RR_\beta^{\Lam+\Lam_i}$, $$
\bar{\iota}_\beta^{\Lambda,i}(p_\beta^{\Lam,i}(a)+[\R[\beta],\R[\beta]])=az+[\RR_\beta^{\Lam+\Lam_i},\RR_\beta^{\Lam+\Lam_i}] ,
$$
and $\Ann_{\RSS^{\Lam+\Lam_i}_\beta}(z(i,\beta)) = \Ker (p_\beta^{\Lam,i})$.
\end{thm}

\begin{proof} Set $\widetilde{\Lam}:=\Lam+\Lam_i$. Note that $t_{\Lam,\beta}(e(\mu)xe(\nu))=0$ whenever $\mu\neq\nu\in I^\beta$. In view of Lemma \ref{lem:injdual}, to prove the theorem, it suffices to show that for any $\nu \in I^\beta$ and $a \in e(\nu) \RSS_\beta^{\widetilde{\Lambda}}e(\nu)$,
 \begin{equation}\label{eq:traiden}
    t_{\widetilde{\Lambda},\beta}(az(i,\beta)) = t_{\Lambda ,\beta}(p_{\beta+\alpha_i}^{\Lam,i}(a)).
  \end{equation}

We show this by induction on $n = |\beta|$. Suppose $\beta = \alpha_j$ for some $j\in I$. Then we have
  \begin{enumerate}
    \item[i)] If $j \neq i$, then $I_{\Lam,\beta}=I_{\widetilde{\Lam},\beta}$, hence $\R[\beta]=\RR_\beta^{\widetilde{\Lam}}$. In this case, $p_{\beta}^{\Lambda,i}$ is a natural isomorphism and $z(i,\beta)$ is the identity, we are done;
    \item[ii)] If $j = i$, then $p_{\beta}^{\Lambda,i}$ can be identified with the canonical projection: $$
    \kb [x_1] /(x_1^{r+1}) \to \kb [x_1]/(x_1^r),\qquad x_1^a+(x_1^r)\mapsto x_1^a+(x_1^r),$$
     where $r = \la \alpha_i^\vee , \Lambda \ra$. In this case, $z(i,\beta)=x_1$, the desired equality \eqref{eq:traiden} of traces follows immediate since in this case the trace function $t_{\widetilde{\Lambda},\beta}$ (resp. $t_{\Lambda,\beta}$) sends a polynomial $f(x_1)$ in $x_1$ to its coefficient of $x_1^{r+1}$ (resp. $x_1^r$).
  \end{enumerate}

  Suppose that \eqref{eq:traiden} holds for any $\beta\in Q_n^+$ with $n\geqslant 1$, we want to show that it holds for any $\beta\in Q_{n+1}^+$. Let $\beta \in Q_n^+$, $\nu \in I^\beta$, $j \in I$ and $a \in e(\nu , j)\RSS_{\beta +\alpha_j}^{\widetilde{\Lambda}}e(\nu,j)$.
  \begin{enumerate}
    \item[a)] If $j \neq i$, then $\widetilde{\lam}_j=\lam_j$ and $az(i,\beta +\alpha_j) = e(\nu,j)az(i,\beta)e(\nu,j)$, we have
        \begin{equation*}
          \begin{aligned}
            t_{\widetilde{\Lambda},\beta+\alpha_j}(az(i,\beta + \alpha_j)) & = r(\beta ,j)t_{\widetilde{\Lambda},\beta}(\hat\varepsilon_{j,\widetilde{\lam}}(az(i,\beta+\alpha_j))) \\
            & = r(\beta ,j)t_{\widetilde{\Lambda},\beta}(\hat\varepsilon_{j,\widetilde{\lam}}(a)z(i,\beta)) \\
            & = r(\beta,j)t_{\Lambda,\beta}(p_\beta^{\Lam,i}(\hat\varepsilon_{j,\widetilde{\lam}}(a))) \\
            & = r(\beta,j)t_{\Lambda,\beta}(\hat\varepsilon_{j,\lam}(p_{\beta+\alpha_j}^{\Lam,i}(a))) \\
            & = t_{\Lambda,\beta+\alpha_j}(p_{\beta+\alpha_i}^{\Lam,i}(a)),
          \end{aligned}
        \end{equation*}
        where the first and the last equalities are \cite[(64)]{SVV},
        the second one follows from the $\RR_\beta^{\widetilde{\Lambda}}$-bilinearity of $\hat\varepsilon_{j,\widetilde{\lambda}}$ (see its definition in \cite[Theorem 3.8]{SVV}),
        the third one follows from induction hypothesis, the fourth one follows from Lemma \ref{keyDifflem}.

    \item[b)] If $j = i$, then $\widetilde{\lam}_j=\lam_j+1$ and $az(\beta+\alpha_i) = ax_{n+1} \cdot z(\beta)e(\beta,i)$, we have similarly
        $$t_{\widetilde{\Lambda},\beta+\alpha_i}(az(i,\beta +\alpha_i)) = r(\beta,i)t_{\Lambda,\beta}(p_{\beta}^{\Lam,i}(\hat\varepsilon_{i,\widetilde{\lambda}}(ax_{n+1}))$$
        and
        $$t_{\Lambda,\beta+\alpha_i}(p_{\beta+\alpha_i}^{\Lam,i}(a)) = r(\beta,i)t_{\Lambda,\beta}(\hat\varepsilon_{i,\lambda}(p_{\beta+\alpha_i}^{\Lam,i}(a))).$$
By Lemma \ref{keyDifflem}, we have $p_{\beta}^{\Lam,i}(\hat\varepsilon_{j,\widetilde{\lambda}}(ax_{n+1})) =
\hat\varepsilon_{j,\lambda}(p_{\beta+\alpha_i}^{\Lam,i}(a))$. It follows that (\ref{eq:traiden}) holds in this case.
        \end{enumerate}
This completes the proof of the theorem.
        \end{proof}

\begin{rem} We remark that Theorem~\ref{thm:annker} still holds when
replacing $z(i,\beta), \Lam+\Lam_i$ by $z(\Lambda',\beta), \Lam+\Lam'$ respectively for arbitrary $\Lam'\in P^+$.
One can prove this by applying Theorem~\ref{thm:annker} repeatedly. Alternatively, one can also prove it directly. The key part in the argument is to check
in the case when $i = j$ the following equality of bubbles:
\begin{equation*}
  p_\beta^{\Lam,\Lam'}(B_{+i, \widetilde{\lambda}}^k) = B_{+i, \lambda}^k \qquad \forall\, 0\leqslant k \leqslant n_i.
\end{equation*}
where $n_i = \la \alpha_i^\vee , \Lambda' \ra$.
\end{rem}

Combine Lemma~\ref{lem:injdual} and Theorem~\ref{thm:annker}, we see that the map $\ov{\iota}_\beta^{\Lambda,i}: \RR_\beta^{\Lam}/[\R[\beta],\R[\beta]]\rightarrow
\RR_\beta^{\Lam+\Lam_i}/[\RR_\beta^{\Lam+\Lam_i},\RR_\beta^{\Lam+\Lam_i}]$ is induced by the $(\RSS_\beta , \RSS_\beta)$-bilinear morphism
$$\iota_\beta^{\Lambda_i} : \RSS_\beta \hookrightarrow \RSS_\beta, \qquad x \mapsto z(i,\beta)x. $$
It is obvious that $\iota_\beta^{\Lambda_i}$ sends monomial basis to monomial basis. Therefore, to show the Center Conjecture \ref{conj:CC}, it suffices to find out a monomial basis
$T_\beta^\Lambda$ for the cocenters $\RSS_\beta^\Lambda / [\RSS_\beta^\Lambda,\RSS_\beta^\Lambda]$ as well as a monomial basis
$T_\beta^{\widetilde{\Lambda}}$ for the cocenters $\RSS_\beta^{\widetilde{\Lam}}/ [\RSS_\beta^{\widetilde{\Lam}},\RSS_\beta^{\widetilde{\Lam}}]$, and to show that
the image of $T_\beta^\Lambda$ under $\ov{\iota}_\beta^{\Lambda,i}$ is a subset of $T_\beta^{\widetilde{\Lambda}}$ and hence are $K$-linearly independent. We name this way as ``cocenter approach'' to the Center Conjecture \ref{conj:CC}. In Section~\ref{sec:MBCocen} we shall construct such monomial bases for some special $\beta$ and use the above ``cocenter approach'' to verify the Center Conjecture \ref{conj:CC} for these special $\beta$.

\section{Bases of cyclotomic KLR algebras and their defining ideals} \label{sec:MBCKLR}

In this section we study the defining ideal $I_{\Lambda ,\beta}$ for $\R[\beta]$. We construct a monomial bases for certain bi-weight space $e(\widetilde{\nu})I_{\Lam,\beta}e(\nu)$ of $I_{\Lambda ,\beta}$. As a result,
we recover a result \cite[Theorem 1.5]{HS} on the monomial bases of the  bi-weight space $e(\widetilde{\nu})\R[\beta]e(\nu)$, in particular, a result  on the monomial bases of the cyclotomic KLR algebras $\R[\beta]$ when $\beta$ satisfying (\ref{eq:DR}). Compared to \cite{HS} and \cite{HS2}, our approach to these monomial bases is elementary in the sense that it does not use the Cyclotomic Categorification Conjecture (\cite{KL09},\cite{KK12}). Moreover, throughout this section, we work over an arbitrary domain $\kb$, and do not need to assume that $\kb$ is a field or a Noetherian domain as required in \cite{HS} and \cite{HS2}.

Note that $I_{\Lambda,\beta}$ is a (right) $\kb [x_1,\ldots,x_n]$ submodule of $\RSS_\beta$, we need the following useful lemma to construct monomial basis of $\kb [x_1,\ldots,x_n]$-modules:

\begin{lem}\label{lem:keylem}
1) Let $R$ be a commutative ring, $g(x) \in R[x]$ be a monic polynomial whose leading term has degree $k$, then $R[x] / (g(x))$ has an $R$-basis $\{ 1, x,\ldots, x^{k-1} \}$.

2) Assume for each $1\leq t\leq n$, there is a polynomials $g_t \in R [x_1,\ldots , x_t]$ satisfying that $g_t$ is monic when regarding as a polynomial in $x_t$ and its leading term is of degree $a_t$. Then the canonical image of the set
  $$B_1 := \{ x^{\un{b}} \mid \un{b}\in \NM^n, \, b_t <a_t,\, \forall\, 1\leqslant t \leqslant n \}$$
  form an $R$-basis of $R [x_1,\ldots , x_n]/ (g_1,\ldots , g_n)$.
  Consequently, the ideal $(g_1,\ldots , g_n)\subset R [x_1,\ldots ,x_n]$ has a basis
  $$B_2 := \Bigl\{ x^{\un{b}}g_k \Bigm|\begin{matrix} \text{$\un{b}=(b_1,\cdots,b_n)\in\NM^n$, $1\leq k\leq n$, and $b_t < a_t$ whenever $k< t \leqslant n$}\end{matrix}\Bigr\},$$
where $x^{\un{b}}:=x_1^{b_1}\cdots x_n^{b_n}$ for each $\un{b}$.
\end{lem}

\begin{proof} The part 1) of the lemma is trivial. We now consider the part 2) of the lemma. Set $R':=R[x_1,\ldots , x_{n-1}]/(g_1,\ldots , g_{n-1})$. Note that
  $$R[x_1,\ldots , x_n] / (g_1,\ldots , g_n) \simto R'[x_n] / (g_n),$$
One could argue by induction on $n$ to prove that $B_1$ is an $R$-basis of $R [x_1,\ldots , x_n]/ (g_1,\ldots , g_n)$.

It remains to prove the last statement of the part 2) of the lemma. Recall the anti-lexicographic order on monomials:
  $$x^{\un{b}} \prec x^{\un{c}} \Leftrightarrow \exists 1\leqslant k \leqslant n,\, \text{such that $b_k < c_k$ and $b_t = c_t, \,\forall\, k<t\leqslant n$}.$$
  Then we have
  $$x_1^{b_1}\cdots x_n^{b_n}g_k = c\cdot x_1^{b_1} \cdots x_k^{b_k + a_k}\cdots  x_n^{b_n} + \text{``lower terms''},$$
  where $c\in R^\times$ and ``lower terms'' is a $R$-linear combination of monomials $x^{\un{c}}$ such that $x^{\un{c}} \prec  x^{\un{b}}\cdot x_k^{a_k}$.
  This yields that the elements in $B_2$ are $R$-linearly independent.   Denote by $I$ the $R$-submodule of $R [x_1,\ldots ,x_n]$ generated by $B_2$, then $I\subset (g_1,\ldots ,g_n )$.
On the other hand, it is easy to see that for any $1\leq k\leq n$ and $f\in R[x_1,\cdots,x_n]$ that $fg_k\in I$. In fact, one can easily see that subtracting off suitable $R[x_1,\cdots,x_n]$-linear combination of $g_n,\cdots,g_{t+1}$ from $fg_k$ will produce an element living inside $$
\text{$R$-Span}\Bigl\{ x^{\un{b}}g_k \Bigm|\begin{matrix} \text{$\un{b}=(b_1,\cdots,b_n)\in\NM^n$, and $b_t < a_t$ whenever $k< t \leqslant n$}\end{matrix}\Bigr\}.
$$
This completes the proof of the lemma.
\end{proof}

\begin{dfn}\label{specialnu} Suppose that $\beta = \sum_{j=1}^{p}k_j\alpha_{i_j}\in Q_n^+$, where $i_s \neq i_t \in I$ for any $1\leq s\neq t\leq p$.
Let $\widetilde{\nu} \in I^\beta$ be of the form
$$\widetilde{\nu} = (\ubr{i_1,\ldots,i_1}{k_{1}\,\text{copies}},\ubr{i_2,\ldots,i_2}{k_{2}\,\text{copies}},\ldots , \ubr{i_p,\ldots,i_p}{k_{p}\,\text{copies}}).$$
We define $i_s \prec i_t\, \Leftrightarrow \, s<t$.
\end{dfn}
Let $w\in\mathfrak{S}_n$. Note that $e(\widetilde{\nu})\tau_w$ is independent of the choice of the reduced expressions of $w\in \SG_n$ since there is no triple
$(r , s ,t )$ satisfying $r < s < t$ and $\widetilde{\nu}_r = \widetilde{\nu}_t \neq \widetilde{\nu}_s$.
So we can choose any reduced expression of $w$ to calculate $e(\widetilde{\nu})\tau_w$. In particular, $e(\widetilde{\nu})\tau_w = e(\widetilde{\nu})\tau_u\tau_t$ whenever $w = us_t$ with $\ell (w) = \ell (u) +1$.

For any $\nu,\nu'\in I^\beta$, we define $\mathfrak{S}(\nu,\nu'):=\{w\in\mathfrak{S}_n|w\nu=\nu'\}$.
The following is the main result of this section.

\begin{thm}\label{thm:MBBWS}
Assume that $\Lambda \in P^+$. Let $\beta \in Q_n^+$ and $\widetilde{\nu}\in I^\beta$ be given as in Definition \ref{specialnu}.
For any $\nu \in I^\beta$, we set $g_{\nu,1}^\Lambda = a_{\nu_1}^\Lambda (x_1)$, and for $1\leqslant k < n$, define
  \begin{equation*}
    g^\Lambda_{\nu,k+1} := \begin{cases}
      \del_k (g^\Lam_{\nu ,k}) & \text{if } \nu_k = \nu_{k+1}; \\
      s_k (g^\Lam_{s_k.\nu,k}) & \text{if } \nu_k \succ \nu_{k+1}; \\
      s_k (g^\Lam_{s_k.\nu,k}) \cdot Q_{\nu_k,\nu_{k+1}}(x_k,x_{k+1}) & \text{if } \nu_k \prec \nu_{k+1}.
    \end{cases}
  \end{equation*}
Then  $g_{\nu,k} \in \kb [x_1,\ldots , x_k]$ is a monic polynomial on $x_k$, and as $\kb [x_1,\ldots ,x_n]$-submodule of $e(\widetilde{\nu})\RSS_\beta e(\nu)$, $e(\widetilde{\nu})I_{\Lambda ,\beta}e(\nu)$ is generated by the set
  $$\{ \tau_wg^\Lambda_{\nu,k}e(\nu)\mid 1\leqslant k\leqslant n, \, w\in \SG (\nu,\widetilde{\nu}) \}.$$
\end{thm}

\begin{proof} Since $\Lambda$ is fixed, we shall write $g_{\nu,k} = g^\Lambda_{\nu,k}$ for simplicity.
By an induction on $k$ we can show that $g_{\nu,k} \in \kb [x_1,\ldots , x_k]$ and the leading term of its $x_k$-expansion is monic.

We have the following compatible decompositions:
  \begin{equation*}
    \begin{aligned}
      & \RSS_\beta = \bigoplus_{\mu,\nu\in I^\beta} e(\mu)\RSS_\beta e(\nu) = \bigoplus_{w\in \SG_n,\nu\in I^\beta} \tau_w\kb [x_1,\ldots ,x_n]e(\nu), \\
      & I_{\Lambda,\beta} = \bigoplus_{\mu,\nu\in I^\beta} e(\mu)I_{\Lambda,\beta} e(\nu), \\
      & \RSS_\beta^\Lambda = \bigoplus_{\mu,\nu\in I^\beta} e(\mu)\RSS_\beta^\Lambda e(\nu).
    \end{aligned}
  \end{equation*}
  In particular, $e(\mu)\RSS_\beta^\Lambda e(\nu) \cong e(\mu)\RSS_\beta e(\nu) / e(\mu)I_{\Lambda,\beta} e(\nu)$.
  Each direct summand in these decompositions is a right $\kb [x_1,\ldots ,x_n]$-module.

  By Definition~\ref{cyclotomicKLR} and Proposition~\ref{prop:MBKLR}, $I_{\Lambda ,\beta}$ is generated as a (right) $\kb [x_1,\ldots , x_n]$-module by elements of form $\tau_u a^\Lambda_{\nu_{v^{-1}(1)}}(x_1)\tau_v e(\nu)$, with $\nu$ running through $I^\beta$.
  We can decompose $v$ as $v = w \cdot s_1s_2\cdots s_t$ with $w \in \SG_{\{2,3,\cdots,n\}}$ and $0\leqslant t \leqslant n-1$. We may assume that the preferred decomposition we choose for $v$ is a product of that of $w$ with $s_1s_2\cdots s_{t}$. Then
  $$\tau_u a^\Lambda_{\nu_{v^{-1}(1)}}(x_1)\tau_v e(\nu) = \tau_{u}\tau_w a^\Lambda_{\nu_{t+1}}(x_1)\tau_1\cdots \tau_t e(\nu).$$
 By Proposition~\ref{prop:MBKLR} and the defining relation of $\RSS_\beta$, this can written as a (right) $\kb [x_1,\ldots ,x_n]$-linear combination of some
 elements of the form:
  $$\tau_{z}a^\Lambda_{\nu_k}(x_1)\tau_1 \tau_2 \cdots \tau_k e(\nu), \quad z\in \SG_n, \quad 0\leqslant k \leqslant n-1.$$
For each $1\leqslant l \leqslant k$, we define $\nu^{l,k} := (s_ls_{l+1}\cdots s_{k-1}).\nu$. In particular, $\nu^{1,1}=\nu$ by convention. Then, as a right $\kb [x_1,\ldots , x_n]$-module, $e(\mu) I_{\Lambda,\beta}e(\nu)$ is generated by
  \begin{equation}\label{eq:IntialGene}
    \Bigl\{ \tau_w g_{\nu^{1,k},1}\tau_1\tau_2\cdots \tau_{k-1}e(\nu) \Bigm| 1\leqslant k \leqslant n, \, w\in \SG (\nu^{1,k},\mu) \Bigr\}.
  \end{equation}

  Now, we concentrate on the case where $\mu = \widetilde{\nu}$.
  For $k > 1$, if $\nu_k \prec \nu_1$, equivalently, $\nu^{2,k}_2 \prec \nu_1^{2,k}$, then $\ell (ws_1) = \ell (w) +1$, and
  \begin{equation}\label{eq:lengp1}
    \tau_w g_{\nu^{1,k},1}\tau_1\tau_2\cdots \tau_{k-1}e(\nu) = \tau_w\tau_1 s_1(g_{\nu^{1,k},1})\tau_2\cdots \tau_{k-1}e(\nu) = \tau_{ws_1}g_{\nu^{2,k},2}\tau_2\cdots \tau_{k-1}e(\nu);
  \end{equation}

  If $\nu_k\succ \nu_1$, equivalently, $\nu_{2}^{2,k}\succ \nu_1^{2,k}$, then $\ell (ws_1) = \ell (w) -1$, and
  \begin{equation}\label{eq:lengm1}
    \tau_w g_{\nu^{1,k},1}\tau_1\tau_2\cdots \tau_{k-1}e(\nu) = \tau_{ws_1}\tau_1^2 s_1(g_{\nu^{1,k},1})\tau_2\cdots \tau_{k-1}e(\nu) = \tau_{ws_1}g_{\nu^{2,k},2}\tau_2\cdots \tau_{k-1}e(\nu);
  \end{equation}

  If $\nu_k = \nu_1$, equivalently, $\nu_2^{2,k} = \nu_1^{2,k}$, then $\nu^{2,k} = \nu^{1,k}$. In this case, for $w$ satisfying $\ell (ws_1) = \ell (w) - 1$, we have
  \begin{equation}\label{eq:SingTerm}
    \tau_w g_{\nu^{1,k},1}\tau_1\tau_2\cdots \tau_{k-1}e(\nu) = \tau_w \del_1 (g_{\nu^{1,k},1})\tau_2\cdots \tau_{k-1}e(\nu) = \tau_{w}g_{\nu^{2,k},2}\tau_2\cdots \tau_{k-1}e(\nu),
  \end{equation}
  where the first equality follows from \eqref{eq:TwoSideDem}; while for $w$ satisfying $\ell (ws_1) = \ell (w) + 1$, we have
  \begin{equation}\label{eq:MulTerms}
  \begin{aligned}
    \tau_w g_{\nu^{1,k},1}\tau_1\tau_2\cdots \tau_{k-1}e(\nu) & = \, \tau_w\del_1(g_{\nu^{1,k},1})\tau_2\cdots \tau_{k-1}e(\nu) + \tau_{w}\tau_1g_{\nu^{1,k},1}\tau_2\cdots \tau_{k-1}e(\nu) \\
    & \quad + \tau_{w}\tau_1\del_1(g_{\nu^{1,k},1})(x_1-x_2)\tau_2\cdots \tau_{k-1}e(\nu) \\
    & = \, \tau_wg_{\nu^{2,k},2}\tau_2\cdots \tau_{k-1}e(\nu) + \tau_{ws_1}\tau_2\cdots \tau_{k-1}g_{\nu^{1,k},1}e(\nu) \\
    & \quad + \tau_{ws_1}g_{\nu^{2,k},2}(x_1-x_2)\tau_2\cdots \tau_{k-1}e(\nu),
  \end{aligned}
  \end{equation}
  where the first equality follows from \eqref{eq:KK42} and \eqref{eq:TwoSideDem}.
  By the defining relations of $\RSS_\beta$, we have
  $$(x_1 - x_2)\tau_2\tau_3\cdots \tau_{k-1}e(\nu) = \tau_2\tau_3\cdots \tau_{k-1}e(\nu)(x_1 - x_k) + \sum_{\substack{2\leqslant j < k\\ \nu_j = \nu_k}} \tau_2\tau_3 \cdots \widehat{\tau_j} \cdots \tau_{k-1}e(\nu),$$
  where $\widehat{\tau_j}$ means deleting $\tau_j$.
  Note that by definition and induction on $l$, we have
  $$g_{\nu^{l,k},l} = g_{\nu^{l,j},l} \qquad \text{ for } j\leqslant k \text{ such that } \nu_j = \nu_k ,$$
  since $\nu^{l,k}_t = \nu^{l,j}_t$ for $1\leqslant t \leqslant l$.
  So
  \begin{equation}\label{eq:InduSum}
    \begin{aligned}
      \tau_{ws_1}g_{\nu^{2,k},2}(x_1-x_2)\tau_2\cdots \tau_{k-1}e(\nu) & = \, \tau_{ws_1}g_{\nu^{2,k},2}\tau_2\cdots \tau_{k-1}e(\nu)(x_1-x_k) \\
      & \quad + \sum_{\substack{2\leqslant j < k\\ \nu_j = \nu_k}} \tau_{ws_1}\tau_{j+1}\cdots \tau_{k-1} g_{\nu^{2,j},2}\tau_2\tau_3\cdots \tau_{j-1}e(\nu) \\
      & = \, \textrm{I} + \textrm{II}.
    \end{aligned}
  \end{equation}

  Writing $\tau_{ws_1}\tau_{j+1}\cdots \tau_{k-1}e(\nu^{2,j})$ in terms of monomial basis, and argue by induction on $j$, we see that $\textrm{II}$ is a right $\kb [x_1,\ldots ,x_n]$ combination of $\tau_{u}g_{\nu^{2,t}}\tau_2\tau_3 \cdots \tau_{t -1}e(\nu)$ with $t < k$, $\nu_t = \nu_k$ and $u \in \SG (\nu^{2,k},\widetilde{\nu})$.

  Consider the following set
  $$
  \begin{aligned}
    & \bigl\{ \tau_w g_{\nu, k}e(\nu) \bigm| 1\leqslant k\leqslant 2, \ w\in \SG (\nu,\widetilde{\nu}) \bigr\} \\
    & \qquad \qquad \bigsqcup \bigl\{ \tau_wg_{\nu^{2, k}, 2}\tau_{2}\cdots \tau_{k-1}e(\nu)\bigm| 2<k\leqslant n, \ w\in \SG (\nu^{2, k},\widetilde{\nu}) \bigr\}.
  \end{aligned}
  $$
  Equations \eqref{eq:lengp1}, \eqref{eq:lengm1}, \eqref{eq:SingTerm}, \eqref{eq:MulTerms} and \eqref{eq:InduSum} show that the right $\kb [x_1,\ldots ,x_n]$-module generated by the above subset is the same as the right $\kb [x_1,\ldots ,x_n]$-module generated by the set \eqref{eq:IntialGene}, and the transition matrix between these two sets of $\kb [x_1,\ldots ,x_n]$-generators is upper-unitriangular (under suitable ordering).

  By induction on $1\leqslant t \leqslant n$ and repeating this argument, one can show that as a right $\kb [x_1,\ldots ,x_n]$-module, $e(\widetilde{\nu})I_{\Lambda,\beta} e(\nu)$ is generated by
  $$
  \begin{aligned}
    & \bigl\{ \tau_w g_{\nu, k}e(\nu) \bigm| 1\leqslant k\leqslant t, \ w\in \SG (\nu,\widetilde{\nu}) \bigr\} \\
    & \qquad \qquad \bigsqcup \bigl\{ \tau_wg_{\nu^{t, k}, t}\tau_{t}\cdots \tau_{k-1}e(\nu)\bigm| t<k\leqslant n, \ w\in \SG (\nu^{t, k},\widetilde{\nu}) \bigr\}.
  \end{aligned}
  $$
  Taking $t = n$, we prove the theorem.
\end{proof}

As a result, we obtain the following result which should be equivalent to \cite[Theorem 1.5]{HS}.

\begin{cor}\label{maincor2} For each $1\leq k\leq n$, let $N_k$ denote the degree of the leading term of the $x_k$-expansion of $g_{\nu ,k}^\Lambda$. The (canonical image of the) set
  $$\{  \tau_wx^{\un{a}}e(\nu) \mid w\in \SG (\nu,\widetilde{\nu}),\, \un{a}\in\NM^n, \, a_k < N_k, \, \forall\, 1\leqslant k \leqslant n \}$$
  form a $\kb$-basis of $e(\widetilde{\nu})\RSS_\beta^\Lambda e(\nu)$.
\end{cor}
\begin{proof} This follows from Theorem~\ref{thm:MBBWS} and Lemma~\ref{lem:keylem} by replacing $\kb [x_1,\ldots ,x_n]$ with $\tau_w\kb [x_1,\ldots ,x_n]e(\nu)$ and $g_k$ with $g_{\nu,k}$.
\end{proof}

We also obtain the following result, where its first part is equivalent to \cite[Theorem 1.9]{HS2} in the case of cyclotomic KLR algebras.

\begin{cor}\label{cor:MBDCR}
  Assume that
  \begin{equation}\label{eq:DR}
    \beta = \alpha_1 + \cdots + \alpha_n, \qquad \alpha_i \neq \alpha_j,\, \forall\, 1\leqslant i\neq j \leqslant n
  \end{equation}
  Then $\RSS_\beta^\Lambda$ has a $\kb$-basis
  $$\Bigl\{ \tau_wx^{\un{a}}e(\nu) \Bigm| \nu\in I^\beta,\, w\in \SG_n ,\, \un{a}\in\NM^n,\, a_t < \la \alpha_{\nu_t}^{\vee} , \Lambda \ra - \sum_{\substack{1 \leqslant k < t \\ w(k) < w(t)}}a_{\nu_t,\nu_k} \, \forall\, 1\leqslant t \leqslant n \Bigr\}.$$
  The two-sided ideal $I_{\Lambda ,\beta}$ has a $\kb$-basis
  $$\Biggl\{ \tau_wx^{\un{a}}g^\Lambda_{w,\nu ,k}e(\nu) \Biggm| \begin{array}{c} \nu\in I^\beta,\, w\in \SG_n ,\, \un{a}\in\NM^n \\ a_t < \la \alpha_{\nu_t}^{\vee} , \Lambda \ra - \operatornamewithlimits{\sum}\limits_{\substack{1 \leqslant k < t \\ w(k) < w(t)}}a_{\nu_t,\nu_k} \, \forall\, k< t \leqslant n \end{array} \Biggr\}.$$
  where $g^\Lambda_{w,\nu ,k}$ is the same as $g^\Lambda_{\nu ,k}$ in Theorem~\ref{thm:MBBWS} by setting $\widetilde{\nu} = w.\nu$.
\end{cor}

We also recover the following result of \cite[Theorem 2.34]{HuL} which gives a monomial basis for cyclotomic NilHecke algebras of type $A$.

\begin{cor}\text{\rm (\cite[Theorem 2.34]{HuL})}
  Let $\beta = n\alpha_i$ and set $l = \la \alpha_i^\vee , \Lambda \ra$. Then $\RSS_\beta^\Lambda$ has a basis
  $$\bigl\{ \tau_wx^{\un{a}}e(i^n) \bigm| w\in \SG_n , \, \un{a}\in\NM^n,\, a_t \leqslant l - t  \bigr\}.$$
\end{cor}


\section{Monomial basis for concenters and the Center Conjecture} \label{sec:MBCocen}

Throughout this section, we shall assume that $\beta$ satisfies \eqref{eq:DR} and $\kb$ is a field. Under this assumption we shall construct a monomial basis for the cocenter $\RSS_\beta^\Lambda / [\RSS_\beta^\Lambda , \RSS_\beta^\Lambda]$ and use it to verify the Center Conjecture.

Fix $\gamma = (\gamma_1 , \gamma_2 ,\ldots ,\gamma_n) \in I^\beta$. By assumption \eqref{eq:DR} we can assume without loss of generality that $I = \{ \gamma_1 , \ldots , \gamma_n \}$. Then $\gamma$ induces a total order on $I$ such that the map
$\iota : \{ 1, \ldots , n \} \to I,\, t \mapsto \gamma_t$ is an order-preserving bijection. We use $\iota$ to identify these two sets.

For any $u,v\in \SG_n$, we define $u\prec v$ if and only if \begin{equation}\label{lex}
(u(1),u(2),\ldots ,u(n)) < (v(1),v(2),\ldots ,v(n)), \end{equation}
where ``$\prec$'' is the lexicographic order on $\mathbb{N}^n$. That says, (\ref{lex}) holds if and only if there exists $t \in [1,n]$, such that $u(t) < v(t)$ and $u(k) = v(k)$ whenever $1\leq k<t$.
We use the bijection $\SG_n \to I^\beta,\, u \mapsto u.\gamma$ to identify $\SG_n$ with $I^\beta$. Thus the total order ``$\prec$'' on $\SG_n$ induces a total order ``$\prec_\gamma$'' on $I^\beta$ and
$\gamma$ is minimal element under this order ``$\prec_\gamma$''.

Note that the induced total order ``$\prec_\gamma$'' on $I^\beta$ is in general not the lexicographical order on $I^\beta$. The induced total order ``$\prec_\gamma$'' on $I^\beta$  can be described as follows: $\mu\prec\nu$ if and only if there exists $t\in [1,n]$, such that if $\mu_p=\nu_q=\gamma_k$ for $1\leq k<t$, then $p=q$ and if $\mu_p=\nu_q=\gamma_t$, then $p<q$.
We call $\gamma$ the \emph{initial weight}. Since $\gamma$ is fixed in the whole section, we simply denote it by ``$\prec$'' instead of by ``$\prec_\gamma$''.
For any $\nu\in I^\beta$ and $e \neq w\in\SG_n$, $e(w.\nu)\RSS_\beta e(\nu)\subset [\RSS_\beta ,\RSS_\beta]$ since the idempotents $\{ e(\nu) \mid \nu \in I^\beta\}$ are mutually orthogonal.
Therefore, by Definition \ref{KLR} and Proposition \ref{prop:MBKLR}, $\RSS_\beta / [\RSS_\beta,\RSS_\beta]$ is isomorphic to $P_\beta / C_\beta$, where $P_\beta = \oplus_{\nu \in I^\beta}e(\nu)\RSS_\beta e(\nu)=\oplus_{\nu \in I^\beta}\kb[x_1,\cdots,x_n]e(\nu)$ is the polynomial subalgebra and $C_\beta = P_\beta \cap [\RSS_\beta,\RSS_\beta]$ is generated as a $\kb$-module by some elements of the form:
$$
\begin{aligned}
  & \, e(\nu)\tau_{u^{-1}}x^{\un{a}}\tau_ux^{\un{b}}e(\nu) - e(u.\nu)\tau_u x^{\un{b}}\tau_{u^{-1}}x^{\un{a}}e(u.\nu) \\
  = & \, \tau_{u^{-1}}\tau_ux^{u^{-1}.\un{a} +\un{b}}e(\nu) - \tau_u \tau_{u^{-1}}x^{\un{a}+u.\un{b}}e(u.\nu) \\
  = & \, Q_{u,\nu}x^{u^{-1}.\un{a} +\un{b}}e(\nu) - Q_{u^{-1},u.\nu}x^{\un{a}+u.\un{b}}e(u.\nu) \\
  = & \, Q_{u,\nu}x^{u^{-1}.\un{a} +\un{b}}e(\nu) - u.\bigl( Q_{u,\nu}x^{u^{-1}.\un{a} +\un{b}}e(\nu) \bigr),
\end{aligned}
$$
where \begin{equation}\label{qunu} Q_{u,\nu} := \operatornamewithlimits{\prod}\limits_{\substack{k<t \\ u(k) > u(t)}}Q_{\nu_k,\nu_t}(x_k , x_t).\end{equation}
By the symmetry, we can easily see that $C_\beta$ has a set of $\kb$-linear generators
\begin{equation*}
  G_1 := \bigl\{ Q_{u,\nu}x^{u^{-1}.\un{a} +\un{b}}e(\nu) - u.\bigl( Q_{u,\nu}x^{u^{-1}.\un{a} +\un{b}}e(\nu) \bigr) \bigm| \nu \in I^\beta , \, \nu \prec u.\nu , \un{a} \in \NM^n \bigr\}.
\end{equation*}
The same argument shows that $\RSS_\beta^\Lambda / [\RSS_\beta^\Lambda,\RSS_\beta^\Lambda] \cong P_\beta^\Lambda / C_\beta^\Lambda$, where $C_\beta^\Lambda$ is the image of $C_\beta$ in $\RSS_\beta^\Lambda$.
As $\RSS_\beta^\Lambda / [\RSS_\beta^\Lambda,\RSS_\beta^\Lambda] \cong \RSS_\beta / ([\RSS_\beta,\RSS_\beta]+I_{\Lambda,\beta})$, $\RSS_\beta^\Lambda / [\RSS_\beta^\Lambda,\RSS_\beta^\Lambda]$ is also isomorphic to
$P_\beta / D_\beta^\Lambda$, where $D_\beta^\Lambda = [\RSS_\beta,\RSS_\beta]+I_{\Lambda,\beta}$ is a $\kb$-submodule generated by
$$
 \begin{aligned}
   G'_1 = & \bigl\{ Q_{u,\nu}x^{\un{a}}e(\nu) - u.\bigl( Q_{u,\nu}x^{\un{a}}e(\nu) \bigr) \bigm| \nu\in I^\beta,\, \nu \prec u.\nu , \, \un{a} \in \NM^n \bigr\} \\
   & \qquad \bigsqcup \left\{ g_{e,\nu,k}^\Lambda x^{\un{b}}e(\nu) \Biggm| \begin{array}{c} \nu\in I^\beta,\, \un{b}\in\NM^n,  \,\text{and}\,\, \forall\, k< t \leqslant n, \\ b_t < \la \alpha_{\nu_t}^{\vee} , \Lambda \ra - \operatornamewithlimits{\sum}\limits_{1 \leqslant k < t }a_{\nu_t,\nu_k} \end{array} \right\},
 \end{aligned}
$$
where $g_{e,\nu,k}^\Lambda$ is defined in Corollary~\ref{cor:MBDCR} and $g_{e,\nu,k}^\Lambda = a^\Lambda_{\nu_t}(x_t)Q_{v,\nu}$ with $v = s_1s_2\cdots s_{t -1}$.

\begin{dfn}\label{indecomp0} Let $\nu \in I^\beta$ and $u \in \SG_n$ such that $\nu \prec u.\nu$. We say that $u$ is \emph{decomposable relative to} $\nu$, if $u = u_1u_2$ such that $\ell (u) = \ell (u_1) + \ell (u_2)$, $u_1 \neq e$ and $\nu \prec u_2.\nu$, and call $u = u_1u_2$ a decomposition of $u$ relative to $\nu$. Conversely, if there is no such decomposition, then we say that $u$ is
\emph{indecomposable relative to $\nu$}.
\end{dfn}
In particular, if  $u$ is indecomposable relative to $\nu$, then for any reduced decomposition $u = s_{i_k}s_{i_{k-1}}\cdots s_{i_1}$, we must have $\nu^{(t)} \prec \nu \prec \nu^{(k)}$ for any $1\leqslant t \leqslant k-1$, where  $\nu^{(t)} := s_{i_t}s_{i_{t-1}}\cdots s_{i_1}.\nu$, $t=1,2,\cdots,k$.

If $u$ is decomposable relative to $\nu$ and $u = u_1u_2$ is a decomposition relative to $\nu$, then
\begin{equation}\label{eq:redind}
\begin{aligned}
  & Q_{u,\nu}x^{\un{a}}e(\nu) - u.\bigl( Q_{u,\nu}x^{\un{a}}e(\nu) \bigr) \\
  = & \, Q_{u_2,\nu}u_2^{-1}.\bigl( Q_{u_1,u_2.\nu} \bigr)x^{\un{a}}e(\nu) - u_2.\left( Q_{u_2,\nu}u_2^{-1}.\bigl( Q_{u_1,u_2.\nu} \bigr)x^{\un{a}}e(\nu) \right) \\
  & \quad + Q_{u_1,u_2.\nu}u_2.\bigl( Q_{u_2,\nu}x^{\un{a}}\bigr)e(u_2.\nu) - u_1.\left( Q_{u_1,u_2.\nu}u_2.\bigl( Q_{u_2,\nu}x^{\un{a}}\bigr)e(u_2.\nu) \right).
\end{aligned}
\end{equation}
This yields that $Q_{u,\nu}x^{\un{a}}e(\nu) - u.\bigl( Q_{u,\nu}x^{\un{a}}e(\nu) \bigr)$ can be written down as a linear combination of $Q_{v,\nu}x^{\un{b}}e(\nu) - v.\bigl( Q_{v,\nu}x^{\un{b}}e(\nu) \bigr)$ with $v$ indecomposable relative to $\nu$. Therefore, $C_\beta$ has a set of $\kb$-generators
$$G_2 := \left\{ Q_{u,\nu}x^{\un{a}}e(\nu) - u.\bigl( Q_{u,\nu}x^{\un{a}}e(\nu) \bigr) \Biggm| \begin{array}{c} \nu\in I^\beta, \, u\in \SG_n\setminus\{e\},\, \un{a}\in \NM^n \\ \text{$u$ is indecomposable relative to $\nu$} \end{array}  \right\},$$
where we have used the condition that $Q_{i,j}(u,v)=Q_{j,i}(v,u)$, and $D_\beta^\Lambda$ is generated as $\kb$-module by
$$
 \begin{aligned}
   G'_2 = & \left\{ Q_{u,\nu}x^{\un{a}}e(\nu) - u.\bigl( Q_{u,\nu}x^{\un{a}}e(\nu) \bigr) \Biggm| \begin{array}{c} \nu\in I^\beta, \, u\in \SG_n\setminus\{e\},\, \un{a}\in \NM^n \\ \text{$u$ is indecomposable relative to $\nu$} \end{array} \right\} \\
   & \qquad \bigsqcup \left\{ g_{e,\nu,k}^\Lambda x^{\un{b}}e(\nu) \Biggm| \begin{array}{c} \nu\in I^\beta,\, \un{b}\in\NM^n, \,\text{and}\,\, \forall\, k< t \leqslant n, \\ b_t < \la \alpha_{\nu_t}^{\vee} , \Lambda \ra - \operatornamewithlimits{\sum}\limits_{1 \leqslant p < t}a_{\nu_t,\nu_p}\end{array} \right\}.
 \end{aligned}
$$
The following lemma describes all $u$ which is indecomposable relative to $\nu$.

\begin{lem}\label{indecomp1} Let $\mu, \nu \in I^\beta$ and $u \in \SG_n$ such that $\mu=u.\nu$. Then $u$ is indecomposable relative to $\nu$ if and only if $u = s_k s_{k+1}\cdots s_t$ such that $\nu_k <\nu_{t+1} < \nu_p$ for $k+1 \leqslant p \leqslant t$, that is, $k = \max \{ 1\leqslant s \leqslant t \mid \nu_s < \nu_{t+1} \}$.
  In particular, such $u$ is determined uniquely by $\nu$ and $t$, and thus, for any $1\leq t < n$, there exists at most one such $u$ that is indecomposable relative to $\nu$.
\end{lem}

\begin{proof}
  The sufficiency is obvious, we show the necessity by induction on $n$. For $n =1,2$, the necessity is obvious.

  For $n\geqslant 3$, assume that $u$ is indecomposable relative to $\nu$ and the necessity holds for any positive integer less or equal than $n-1$.
  Let $v,w\in \SG_n$, such that $\nu = v.\gamma$, $\mu = w.\gamma$, then $w = uv$.
  In particular, $\nu_t = v^{-1}(t) \in I$ and $\mu_t = w^{-1}(t) \in I$.
  Since $\nu \prec \mu$, there exists $1\leq t\leq n$, such that $v(k) = w(k)$, $\forall 1\leqslant k < t$ and $v(t) < w(t)$.

  Assume $t > 1$. Let $p = v(1)=w(1)$. If $u = u_1 s_qs_{q-1}\cdots s_p u_2$, where $q\geq p$, $u_2 \in \SG_{[p+1,n]}$, $u_1$ is the distinguished minimal length left coset representative in $u\SG_{[p , n]}$.
  Then $u':=s_p u_2$ satisfies that $v(1) = p < p+1 = u'v(1)$, which implies that $u_1s_qs_{q-1}\cdots s_{p+1} \neq e$ and  $\nu \prec u'.\nu$, contradict to the assumption that $u$ is indecomposable relative to $\nu$. If $u = u_1 u_2$, where $u_2 \in \SG_{[p+1,n]}$, $u_1$ is the distinguished minimal length left coset representative in $u\SG_{[p , n]}$. Then we must have $u_1 \in \SG_{p-1}$ because $p = v(1) = u(p) = u_1(p) < u_1(p+1) < \cdots < u_1(n)$. Hence $u = u_1u_2 = u_2u_1$. Furthermore, if $v(t) < p$, then for any $1\leq k<t$, $$
  u_1v(k)=\begin{cases} u_1u_2v(k)=uv(k)=w(k)=v(k), &\text{if $v(k)<p$;}\\
  v(k), &\text{if $v(k)>p$,}
  \end{cases}
  $$
  while $u_1u_2v(t)=u(v(t))=w(t)>v(t)$. Hence $\nu \prec u_1.\nu$. In a similar way we can prove that if $v(t) > p$ then $\nu \prec u_2.\nu$.  Now $u$ is indecomposable relative to $\nu$ implies that in the former case we have $u_2 = e$ and $u =u_1$, while in the latter case $u_1 =e$ and $u = u_2$. In both case we can apply the induction hypothesis to deduce that $u$ is of the desired form.

  Assume $t =1$. Without loss of generality we may assume that $\mu_n \neq \nu_n$, since otherwise $u \in \SG_{n-1}$, and $u$ is of the desired form by induction hypothesis. Therefore, we can write the element $u$ uniquely as $u = s_ks_{k+1}\cdots s_{n-1}u_1$, where $u_1 \in \SG_{n-1}$ and $k \leqslant n-1$.
  It remains to show that $u_1=e$. Suppose this is not the case, i.e., $u_1 \neq e$. Set $u' := s_k u$. Then $u' \neq e$.
  Since $u$ is indecomposable relative to $\nu$, we have $u'.\nu \prec \nu \prec u.\nu$.
  We can deduce that $u'v(1) \in \{ k , k+1 \}$, because otherwise $u'v(1) = w(1) > v(1)$ which contradicts to $u'v.\gamma \prec v.\gamma$.
  Now $u'v.\gamma \prec w.\gamma$ implies that $w(1) = k +1$. That is, $u'v(1) = k = u_1v(1)$.
  Because $u_1.\nu \prec \nu$, $k = u_1v(1) \leqslant v(1) < w(1) = k+1$, we have $v(1) = u_1v(1) = k$.
  Using a similar argument as in the case $t > 1$, we can easily obtain that $u_1 = u'_1 u'_2$, with $u'_1 \in \SG_{[1,k-1]}$ and $u'_2 \in \SG_{[k+1,n-1]}$.
  We claim that $u'_1 =e$ since otherwise $u = u'_1 \cdot (s_ks_{k+1}\cdots s_{n-1}u'_2)$ is a decomposition of $u$ relative to $\nu$, contradiction.
  On the other hand, for each $k+1\leq j\leq n-2$, $s_ks_{k+1}\cdots s_{n-1}s_j s_{n-1}s_{n-2}\cdots s_{k}=s_{j+1}$. It follows that $t = s_ks_{k+1}\cdots s_{n-1}u'_2 s_{n-1}s_{n-2}\cdots s_{k}$ satisfies that $\ell (t) = \ell (u'_2)$. Thus $u'_2 \neq e$ implies that $t \cdot (s_k s_{k+1}\cdots s_{n-1})$ is a decomposition of $u$ relative to $\nu$, which is a contradiction. In summary, $u'_1 = e = u'_2$ and $u = s_ks_{k+1}\cdots s_{n-1}$.

  Finally, if $u = s_k s_{k+1}\cdots s_t$ is indecomposable relative to $\nu$, then it is easy to check that $\nu_k <\nu_{t+1} < \nu_p$ for $k+1 \leqslant p \leqslant t$. We are done.
\end{proof}

\begin{remark}
  The notion of ``indecomposable relative $\nu$'' can be easily understood using the diagrammatic presentation of the symmetric group $\SG_n$ (see, for example \cite[Exercise 1.5]{soergelbook}).  A reduced expression of $u$ gives a strand diagram in $\RM \times (0,1)$ matching $\nu$ and $\nu' = u.\nu$.
  For example, $u =\begin{pmatrix}1,2,3,4\\ 4,1,3,2\end{pmatrix}=s_2s_3s_2s_1\in\SG_4$, $\nu = (\nu_1,\nu_2,\nu_3,\nu_4) = (1,4,3,2)$. Then $\nu=(1,4,3,2)\prec u.\nu=(4,2,3,1)$:
  
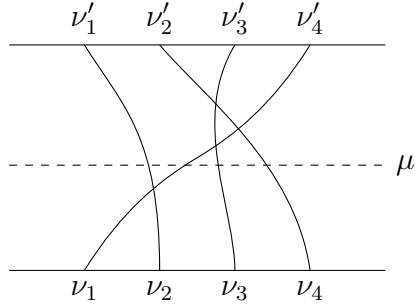
\begin{equation*}
    \begin{tikzpicture}
      \draw (0,0) edge (5,0);
      \draw (0,3) edge (5,3);
      \draw [dashed] (0,1.4) edge (5,1.4);
      \draw (1,0) node [anchor=north,color=black] {$\nu_1$};
      \draw (2,0) node [anchor=north,color=black] {$\nu_2$};
      \draw (3,0) node [anchor=north,color=black] {$\nu_3$};
      \draw (4,0) node [anchor=north,color=black] {$\nu_4$};
      \draw (1,3) node [anchor=south,color=black] {$\nu'_1$};
      \draw (2,3) node [anchor=south,color=black] {$\nu'_2$};
      \draw (3,3) node [anchor=south,color=black] {$\nu'_3$};
      \draw (4,3) node [anchor=south,color=black] {$\nu'_4$};
      \draw (5,1.4) node [anchor=west,color=black] {$\mu$};
      \draw (1,0) .. controls (1.5,0.8) and (2,1.2) .. (2.5,1.5) .. controls (3,1.8) and (3.5,2.2) .. (4,3);
      \draw (2,0) .. controls (2,1.8) and (1.5,2.2) .. (1,3);
      \draw (3,0) .. controls (3,0.8) and (2.4,2) .. (3,3);
      \draw (4,0) .. controls (3.8,1.4) and (2.6,2.3) .. (2,3);
    \end{tikzpicture}
  \end{equation*}
The element $u$ is decomposable relative to $\nu$ is equivalent to that we can find a strand diagram representing $u$ and a horizontal line $\RM \times \{ c \}$ with $0 < c <1$, such that their intersection $\mu$ satisfies $\nu \prec \mu$. In particular, the above diagram shows that $w$ is decomposable relative to $\nu$.

Let $\nu$ as above. Now we consider the element $v  = \begin{pmatrix}1,2,3,4\\ 2,3,4,1\end{pmatrix}=s_1s_2s_3\in\SG_4$ and $\nu'' = v.\nu=(2,1,4,3)$, then $v = s_1s_2s_3$ is the unique reduced decomposition of $v$ and
  
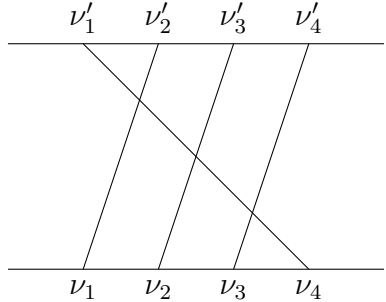
\begin{equation*}
    \begin{tikzpicture}
      \draw (0,0) edge (5,0);
      \draw (0,3) edge (5,3);
      \draw (1,0) node [anchor=north,color=black] {$\nu_1$};
      \draw (2,0) node [anchor=north,color=black] {$\nu_2$};
      \draw (3,0) node [anchor=north,color=black] {$\nu_3$};
      \draw (4,0) node [anchor=north,color=black] {$\nu_4$};
      \draw (1,3) node [anchor=south,color=black] {$\nu'_1$};
      \draw (2,3) node [anchor=south,color=black] {$\nu'_2$};
      \draw (3,3) node [anchor=south,color=black] {$\nu'_3$};
      \draw (4,3) node [anchor=south,color=black] {$\nu'_4$};
      \draw (1,0) edge (2,3);
      \draw (2,0) edge (3,3);
      \draw (3,0) edge (4,3);
      \draw (4,0) edge (1,3);
    \end{tikzpicture}
  \end{equation*}
  it is easy to check that $v$ is indecomposable relative to $\nu$. One can also check that $v$ is decomposable relative to $\mu = (1,4,2,3)$.
\end{remark}

\begin{cor}
  If $u = s_k s_{k+1}\cdots s_t \in \SG_n$ is indecomposable relative to $\nu$, then $Q_{u,\nu} = \operatornamewithlimits{\prod}\limits_{p=k}^t Q_{\nu_p,\nu_{t+1}}(x_p , x_{t+1})$ is monic as polynomial of $x_{t+1}$.
\end{cor}

Now, we can construct the monomial basis for cocenter $\RSS_\beta^\Lambda / [\RSS_\beta^\Lambda,\RSS_\beta^\Lambda]$.

\begin{thm}\label{thm:MBCocen}
  Assume that $\beta$ satisfies \eqref{eq:DR}, and fix an initial weight $\gamma \in I^\beta$. For each $\nu\in I^\beta$ and $1\leq t\leq n$, we define $k_{\nu,t}:=\max\{1\leq p<t|\nu_p<\nu_t\}$; or $0$ if $\{1\leq p<t|\nu_p<\nu_t\}=\emptyset$. Set
  $$
  T_\gamma^\Lambda := \Biggl\{ x^{\un{a}}e(\nu) \Biggm| \begin{matrix}\text{$\nu \in I^\beta, \un{a}\in\NM^n$, and for any $1\leq t\leq n$,}\\
  \text{$a_t < - \operatornamewithlimits{\sum}\limits_{l = \max\{k_{\nu,t},1\}}^{t-1} a_{\nu_t , \nu_l} + \delta_{k_{\nu,t}, 0}\la \alpha_{\nu_t}^\vee , \Lambda \ra$} \end{matrix} \Biggr\}.
  $$
  Then the canonical image of $T_\gamma^\Lambda$ in $\RSS_\beta^\Lambda / [\RSS_\beta^\Lambda,\RSS_\beta^\Lambda]$ gives a $\kb$-basis of $\RSS_\beta^\Lambda / [\RSS_\beta^\Lambda,\RSS_\beta^\Lambda]$.
\end{thm}

\begin{proof} Recall the definition of $G'_2$ in the paragraph above Lemma \ref{indecomp1}, where we have decomposed $G'_2$ as a disjoint union
$G'_2 = D_2 \sqcup M_2$, such that $D_2$ consists of those commutators and $M_2$ is a basis of $P_\beta \cap I_{\Lambda ,\beta}$ by Corollary~\ref{cor:MBDCR}.
 We are going to pick out a subset $D'_2$ of $D_2$ such that $T^\Lambda_\gamma \sqcup D'_2 \sqcup M_2$ form a $\kb$-basis of $\RSS_\beta$ and $G'_2$ is generated by $D'_2 \sqcup M_2$. Once this is done,
 it is clear that the canonical image of $T_\gamma^\Lambda$ in $\RSS_\beta^\Lambda / [\RSS_\beta^\Lambda,\RSS_\beta^\Lambda]$ gives a $\kb$-basis of $\RSS_\beta^\Lambda / [\RSS_\beta^\Lambda,\RSS_\beta^\Lambda]$ because
$D'_2 \sqcup M_2$ is a $\kb$-basis of $[\RSS_\beta,\RSS_\beta]+I_{\Lam,\beta}$.

  For any $\nu \in I^\beta$, $u\in\SG_n$ which is
  indecomposable relative to $\nu$ and $\un{a}\in\NM^n$, we set $P(u,\nu ,\un{a}) := Q_{u,\nu}x^{\un{a}}e(\nu) - u.\bigl( Q_{u,\nu}x^{\un{a}}e(\nu) \bigr)$.
  Define a subset $D'_2$ of $D_2$ as follows:
  $$
  \begin{aligned}
    & D'_2 = \left\{ P(u,\nu ,\un{a})\in D_2 \Biggm| \begin{array}{c} \nu \in I^\beta ,\, \un{a} \in \NM^n ,\, u = s_ps_{p+1}\cdots s_{q-1}, \\ \text{$u$ is
    indecomposable relative to $\nu$,} \\ a_t <  d_{\nu ,u ,t} \; \forall\, 1\leqslant t \leqslant n \end{array} \right\}
  \end{aligned}
  $$
  where $d_{\nu,u,t}$ (for $u =s_ps_{p+1}\cdots s_{q-1}$) is defined by
  $$
  d_{\nu,u,t} = \begin{cases}
    - \operatornamewithlimits{\sum}\limits_{k = \max (k_{\nu ,t} , 1)}^{t-1} a_{\nu_t , \nu_k} + \delta_{k_{\nu,t} , 0}\la \alpha_{\nu_t}^\vee , \Lambda \ra &
    \text{if } t > q; \\
    \la \alpha_{\nu_t}^\vee , \Lambda \ra - \operatornamewithlimits{\sum}\limits_{k=1}^{p-1} a_{\nu_t,\nu_k} & \text{if } t = q;\\
    \la \alpha_{\nu_t}^{\vee} , \Lambda \ra - \operatornamewithlimits{\sum}\limits_{k=1}^{t-1} a_{\nu_t,\nu_k}  & \text{if } 1\leqslant t < q.
  \end{cases}
  $$

  \emph{Step 1}. We claim that $T^\Lam_\gamma \sqcup D'_2\sqcup M_2$ is a basis of $P_\beta$. Recall that by Proposition~\ref{prop:MBKLR}, $B_\beta = \{ x^{\un{a}}e(\nu)\mid \nu \in I^\beta, \un{a}\in\NM^n \}$ is
  a $\kb$-basis of $P_\beta$.

  We define an total order $\prec_\gamma$ on $B_\beta$ as follows: $x^{\un{a}}e(\nu) \prec_\gamma x^{\un{b}}e(\mu)$ if either $\nu \prec_\gamma \mu$, or $\nu = \mu$ and there is $k\in [1,n]$ s.t. $a_t = b_t$ for $t > k$ and $a_k > b_k$.
  The assumption (\ref{eq:DR}) implies that $g^\Lambda_{e,\nu , k}e(\nu)=a_{\nu_k}^\Lam(x_k)\prod_{p=1}^{k-1}Q_{\nu_p,\nu_k}(x_p,x_k)e(\nu)$.
  Then, for $P(u,\nu , \un{a}) \in D_2$ and $g^\Lambda_{e,\nu , k}x^{\un{b}}e(\nu) \in M_2$, we have
  $$
  \begin{aligned}
    & P(u,\nu , \un{a}) = c \cdot x^{\un{a}}\cdot x_t^{-\operatornamewithlimits{\sum}\limits_{p=k}^{t-1}a_{\nu_t,\nu_p}}e(\nu) + \text{``higher terms'',} \qquad \text{if $u = s_ks_{k+1}\cdots s_{t-1}$}, \\
    & g^\Lambda_{e,\nu , k}x^{\un{b}}e(\nu) = c'\cdot x^{\un{b}} \cdot x_k^{- \operatornamewithlimits{\sum}\limits_{p=1}^{k-1} a_{\nu_k , \nu_p} + \la \alpha_{\nu_k}^\vee , \Lambda \ra}e(\nu) + \text{``higher terms''}.
  \end{aligned}
  $$
  where $c,c'$ are invertible and ``higher terms'' means a $\kb$-linear combination of monomials $x^{\un{d}}e(\mu)\in B_\beta$ that is larger than the ``leading term'' $x^{\un{a}}\cdot x_t^{-\operatornamewithlimits{\sum}\limits_{p=k}^{t-1}a_{\nu_t,\nu_p}}e(\nu)$ and $x^{\un{b}} \cdot x_k^{- \operatornamewithlimits{\sum}\limits_{p=1}^{k-1} a_{\nu_k , \nu_p} + \la \alpha_{\nu_k}^\vee , \Lambda \ra}e(\nu)$ under the order $\prec_\gamma$ respectively.
  For any $x^{\un{a}}e(\nu) \in B_\beta$, one can find a unique element $f \in T_\gamma^\Lambda \sqcup D'_2\sqcup M_2$, s.t. $f = c\cdot x^{\un{a}}e(\nu) +
  \text{``higher terms''}$. This implies the claim of this step.

  \emph{Step 2}. We claim that $D_\beta^\Lambda$ is generated as a $\kb$-module by $D'_2\sqcup M_2$. Suppose this is not the case. Then we can find $P(u,\nu ,\un{a}) \in D_2\backslash D'_2$, which does not belong to the $\kb$-module generated by $D'_2\sqcup M_2$.
  As $B_\beta$ is upper bounded under $\prec_\gamma$, we can find such a $P(u,\nu ,\un{a})$ whose leading term is maximal under $\prec_\gamma$.
  By the definition of $D'_2$, there is some $t\in [1,n]$, such that $a_t \geqslant d_{\nu ,u,t}$ and we can choose $t$ such that it is as maximal as possible.
  Now, by Lemma~\ref{indecomp1} , $u$ is form of $s_ps_{p+1}\cdots s_{q-1}$.

  (1) If $1\leqslant t<q$, we have
  \begin{equation}
    \begin{aligned}
      P(u,\nu,\un{a}) = & \, c\cdot \left(Q_{u,\nu}g^\Lambda_{e,\nu , t}x^{\un{a}}\cdot x_t^{-d_{\nu ,u,t}}e(\nu) - u.\bigl( Q_{u,\nu}g^\Lambda_{e,\nu , t}x^{\un{a}}\cdot x_t^{-d_{\nu ,u,t}}e(\nu)\bigr)\right) \\
      & + \text{``higher terms''}
    \end{aligned}
  \end{equation}
  where $c\in \kb^\times$ and ``higher term'' is a $\kb$-linear combination of some $P(u ,\nu ,\un{b})$ whose leading terms are larger than that of
  $P(u,\nu ,\un{a})$.
  By our assumption, ``higher terms'' belong to the $\kb$-submodule generated by $D'_2\sqcup M_2$.
  Note that $Q_{u,\nu}g^\Lambda_{e,\nu , t}x^{\un{a}}\cdot x_t^{-d_{\nu ,u,t}}e(\nu)$ belong to the $\kb$-submodule generated by $M_2$ and so does $$
  u.\bigl( Q_{u,\nu}g^\Lambda_{e,\nu , t}x^{\un{a}}\cdot x_t^{-d_{\nu ,u,t}}e(\nu)\bigr) = \tau_u \tau_{t-1}\tau_{t-2}\cdots \tau_1 a_{\nu_t}^\Lambda
  (x_1)\tau_1\tau_2\cdots \tau_{t-1}\tau_{u^{-1}}e(u.\nu). $$
  This implies that $P(u,\nu ,\un{a})$ belongs to the $\kb$-module generated by $D'_2\sqcup M_2$, contradiction.

  (2) If $t =q$, then as $g^\Lambda_{e,\nu ,q} = Q_{u,\nu}\cdot a_{\nu_q}^\Lambda (x_q) \cdot
  \operatornamewithlimits{\prod}\limits_{k=1}^{p-1}Q_{\nu_k,\nu_q}(x_k,x_q)$, we have
  \begin{equation}
    \begin{aligned}
      P(u,\nu,\un{a}) = & \, c\cdot \left( g^\Lambda_{e,\nu ,q}x^{\un{a}}\cdot x_q^{-d_{\nu,u,q}}e(\nu) - u.(g^\Lambda_{e,\nu ,q}x^{\un{a}}\cdot x_q^{-d_{\nu,u,q}}e(\nu)) \right) + \text{``higher terms''},
    \end{aligned}
  \end{equation}
  where $c\in \kb^\times$ and ``higher term'' is a $\kb$-combination of $P(u ,\nu ,\un{b})$ whose leading terms are larger than that of $P(u,\nu ,\un{a})$.
  A similar argument as in case (1) shows that both $g^\Lambda_{e,\nu ,q}x^{\un{a}}\cdot x_q^{-d_{\nu,u,q}}e(\nu)$ and $u.(g^\Lambda_{e,\nu ,q}x^{\un{a}}\cdot x_q^{-d_{\nu,u,q}}e(\nu))$ belong to the $\kb$-module generated by $M_2$ and this implies that $P(u,\nu ,\un{a})$ belongs to the $\kb$-module generated by $D'_2\sqcup M_2$, contradiction.

  (3) If $t>q$ and the set $A_{\nu ,t} = \{ 1\leqslant i <t \mid \nu_i <\nu_t \} \neq \emptyset$, then we set $k = \max A_{\nu,t}$,
  \begin{equation}
  \begin{aligned}
    P(u,\nu,\un{a}) & = \, c\cdot \Bigl( Q_{u,\nu}Q_{v,\nu}x^{\un{a}}\cdot x_{t}^{-d_{\nu ,u,t}} e(\nu) - u.\bigl( Q_{u,\nu}Q_{v,\nu}x^{\un{a}}\cdot
x_{t}^{-d_{\nu ,u ,t}}  e(\nu) \bigr) \Bigr) \\
    & \qquad + \text{``higher terms''},
  \end{aligned}
  \end{equation}
  where $v = s_ks_{k+1}\cdots s_{t-1}$ and ``higher terms'' means a $\kb$-linear combination of $P(u ,\nu ,\un{b})$ whose leading terms are larger than that of
$P(u,\nu ,\un{a})$.
  By assumption, these $P(u ,\nu ,\un{b})$ belong to the $\kb$-module generated by $D'_2\sqcup M_2$.
  On the other hand, set $x^{\un{c}} = x^{\un{a}}\cdot x_{t}^{-d_{\nu ,u ,t}}$, the assumption on $A_{\nu,t}$ yields that $v$ is indecomposable relative to
$\nu$ and we have
  \begin{equation}
    \begin{aligned}
      Q_{u,\nu}Q_{v,\nu}x^{\un{c}} e(\nu) - u.\bigl( Q_{u,\nu}Q_{v,\nu}x^{\un{c}} e(\nu) \bigr) & = \, Q_{u,\nu}Q_{v,\nu}x^{\un{c}} e(\nu) - v.\bigl( Q_{u,\nu}Q_{v,\nu}x^{\un{c}} e(\nu) \bigr) \\
      & \quad + v.\bigl( Q_{u,\nu}Q_{v,\nu}x^{\un{c}} e(\nu) \bigr) - u.\bigl( Q_{u,\nu}Q_{v,\nu}x^{\un{c}} e(\nu) \bigr) 
    \end{aligned}
  \end{equation}
  It is easy to check that $\ell (uv^{-1}) = \ell (u) + \ell (v^{-1})$, and by definition, we have
  $$
  \begin{aligned}
    & v.\bigl( Q_{u,\nu}Q_{v,\nu}x^{\un{c}} e(\nu) \bigr) = \tau_v\tau_{u^{-1}}\tau_u\tau_{v^{-1}}x^{v.\un{c}}e(v.\nu) = Q_{uv^{-1},v.\nu}x^{v.\un{c}}e(v.\nu) \\
    & u.\bigl( Q_{u,\nu}Q_{v,\nu}x^{\un{c}} e(\nu) \bigr) = \tau_{u}\tau_{v^{-1}}\tau_{v}\tau_{u^{-1}}x^{u.\un{c}}e(u.\nu) = Q_{vu^{-1},u.\nu}x^{u.\un{c}}e(u.\nu) \\
    & v.\bigl( Q_{u,\nu}Q_{v,\nu}x^{\un{c}} e(\nu) \bigr) - u.\bigl( Q_{u,\nu}Q_{v,\nu}x^{\un{c}} e(\nu) \bigr) \\
    & \qquad  \qquad \qquad \qquad = Q_{uv^{-1},v.\nu}x^{v.\un{c}}e(v.\nu) - uv^{-1}.\bigl( Q_{uv^{-1},v.\nu}x^{v.\un{c}}e(v.\nu) \bigr)
  \end{aligned}
  $$
  Now, as $\nu \prec_\gamma v.\nu$ and $\nu\prec_\gamma u.\nu$, \eqref{eq:redind} implies that $ v.\bigl( Q_{u,\nu}Q_{v,\nu}x^{\un{c}} e(\nu) \bigr) -
u.\bigl( Q_{u,\nu}Q_{v,\nu}x^{\un{c}} e(\nu) \bigr)$ is a $\kb$-linear combination of $P(w,\mu , \un{b})$ with $\nu \prec_\gamma \mu$ and hence belongs to the
$\kb$-submodule generated by $D'_2\sqcup M_2$ by our assumption.

  On the other hand, $Q_{u,\nu}Q_{v,\nu}x^{\un{c}} e(\nu) - v.\bigl( Q_{u,\nu}Q_{v,\nu}x^{\un{c}} e(\nu) \bigr)$ is a $\kb$-linear combination of some $P(v,\nu ,\un{b})$ and the one having minimal leading term among those appear with non-zero coefficients is $P(v,\nu , \un{d})$ satisfying that $x^{\un{d}}e(\nu) = x^{\un{c}}x_q^{-\operatornamewithlimits{\sum}\limits_{i = p}^{q-1}a_{\nu_q ,\nu_i}}e(\nu)$.
 Note that the leading term of $P(u,\nu ,\un{a})$ is equal to that of $P(v,\nu , \un{d})$ and this yields that $P(v,\nu ,\un{d})$ does not belong to the $\kb$-module generated by $D'_2\sqcup M_2$ by our assumption.
  But for $t' > t > q$, we have $d_{\nu,v,t'} = d_{\nu ,u,t'} = a_{t'} = d_{t'}$.
  As by assumption $P(v,\nu ,\un{d}) \not\in D'_2$, we can find $p\in [1,n]$ as large as possible such that $d_{p} \geqslant d_{\nu ,v ,p}$. The argument above
shows that $p\leqslant t$. Therefore, we are in a position to apply the case (1) (if $1\leqslant p < t$); or the case (2) (if $p =t$) to $P(v,\nu,\un{d})$.
In both case, we obtain a contradiction.

  (4) If $t > q$ and $A_{\nu ,t} = \{ 1\leqslant i <t \mid \nu_i <\nu_t \} = \emptyset$, then we have
  \begin{equation}
    \begin{aligned}
      P(u,\nu,\un{a}) = & \, c\cdot \left(Q_{u,\nu}g^\Lambda_{e,\nu , t}x^{\un{a}}\cdot x_t^{-d_{\nu ,u,t}}e(\nu) - u.\bigl( Q_{u,\nu}g^\Lambda_{e,\nu , t}x^{\un{a}}\cdot x_t^{-d_{\nu ,u,t}}e(\nu)\bigr)\right) \\
      & + \text{``higher terms''}
    \end{aligned}
  \end{equation}
  where $c\in \kb^\times$ and ``higher term'' is a $\kb$-linear combination of $P(u ,\nu ,\un{b})$ whose leading terms are larger than that of $P(u,\nu ,\un{a})$.
  An argument similar to the case (1) yields a contradiction.

  \emph{Step 3}. $D'_2\sqcup M_2$ is a $\kb$-basis of $D_\beta^\Lambda$. This is a immediate consequence of Step 1 and 2.
  Moreover, the image of $T_\gamma^\Lambda$ is a basis of $\RSS_\beta^\Lambda / [\RSS_\beta^\Lambda,\RSS_\beta^\Lambda]$.
  We're done.
\end{proof}

\begin{lem}\label{finalInject}
  Assume that $\beta \in Q_n^+$ satisfies \eqref{eq:DR}. For $i\in I$ and $\Lambda \in P^+$, set $\widetilde{\Lambda} = \Lambda + \Lambda_i$. Then the $\kb$-linear
homomorphism $\ov{\iota}_{\beta}^{\Lambda,i}$ is injective. Equivalently, the following natural algebra homomorphism
  $$\ov{p}_\beta^{\Lambda ,i} : Z(\RSS_\beta^{\widetilde{\Lambda}}) \to Z(\RSS_\beta^\Lambda)$$
  is surjective.
\end{lem}

\begin{proof}
  Pick $\gamma \in I^\beta$ such that $\gamma_1 = i$. Let $x^{\un{a}}e(\nu) \in T_\gamma^\Lambda$. Then
  $$z(i,\beta)x^{\un{a}}e(\nu) = x^{\un{b}}e(\nu), $$
  where $b_t = a_t + \delta_{\nu_t,i}$.
  For $\nu \in I^\beta$, there is at most one $t\in [1,n]$, such that $\nu_t = i$, and in this case,
 $$
    \begin{aligned}
      & a_t < - \operatornamewithlimits{\sum}\limits_{l =1}^{t-1} a_{\nu_t , \nu_l} + \la \alpha_{\nu_t}^\vee , \Lambda \ra ,\\
      & b_t < - \operatornamewithlimits{\sum}\limits_{l = 1}^{t-1} a_{\nu_t , \nu_l} + \la \alpha_{\nu_t}^\vee ,
\Lambda \ra + 1 =
 - \operatornamewithlimits{\sum}\limits_{l = 1}^{t-1} a_{\nu_t , \nu_l} + \la \alpha_{\nu_t}^\vee , \widetilde{\Lambda} \ra
    \end{aligned}
$$
  and for any $1\leq p\leq n$ with $p\neq t$, we have   $$
    \begin{aligned}
      & a_p < - \operatornamewithlimits{\sum}\limits_{l = \max\{k_{\nu,p},1\}}^{p-1} a_{\nu_p , \nu_l} + \delta_{k_{\nu,p} , 0}\la \alpha_{\nu_p}^\vee ,
\Lambda \ra ,\\
      & b_p < - \operatornamewithlimits{\sum}\limits_{l = \max\{k_{\nu,p},1\}}^{p-1} a_{\nu_p , \nu_l} + \delta_{k_{\nu,p} , 0}\la \alpha_{\nu_p}^\vee ,
\Lambda \ra  = - \operatornamewithlimits{\sum}\limits_{l =\max\{k_{\nu,p},1\}}^{p-1} a_{\nu_p , \nu_l} + \delta_{k_{\nu,p} , 0}\la \alpha_{\nu_p}^\vee , \widetilde{\Lambda} \ra
    \end{aligned}
 $$
  So, $x^{\un{b}}e(\nu) \in T_\gamma^{\widetilde{\Lambda}}$, and this completes the proof.
\end{proof}

\begin{cor}
  Assume that $\beta \in Q_n^+$ satisfies \eqref{eq:DR}. For $\Lambda ,\Lambda' \in P^+$, set $\widetilde{\Lambda} = \Lambda + \Lambda'$.
  The canonical homomorphism of algebras
  $$Z (\RSS_\beta^{\widetilde{\Lambda}}) \to Z (\RSS_\beta^\Lambda)$$
  is surjective.
\end{cor}

Now, by Theorem~\ref{thm:CC} or \ref{thm:CC2}, we have

\begin{thm}\label{thm:cenconj}
  Assume that $\beta \in Q_n^+$ satisfies \eqref{eq:DR}. For any $\Lambda \in P^+$, the canonical algebra homomorphism
  $$Z(\RSS_\beta) \to Z (\RSS_\beta^\Lambda)$$
  is surjective. In other words, the Center Conjecture \ref{conj:CC} hods for $\R[\beta]$.
\end{thm}

\vskip 3em


\bibliographystyle{amsplain}


\end{document}